%% file: instance_goafem.tex
\documentclass[a4paper,12pt]{amsart}

\usepackage{fullpage}
\usepackage{amsmath,amssymb,graphicx,psfrag}
\usepackage[T1]{fontenc}
\usepackage{amsmath,amssymb,ifthen}
\usepackage{color}
\usepackage{moreverb}
\usepackage{enumitem}
\usepackage{pgfplots}
\usepackage{adjustbox}
\usepackage{algorithmic}
\usepackage{mathtools}
\usepackage[width=\textwidth,textfont=it]{caption}
\usepackage{subcaption}

\newcommand{\enorm}[2][]{#1|\!#1|\!#1|\,#2\,#1|\!#1|\!#1|}
\newcommand{\norm}[3][]{#1\|#2#1\|_{#3}}
\newcommand{\set}[3][\big]{#1\{#2\,:\,#3#1\}}

\def\reff#1#2{\!\stackrel{\eqref{#1}}{#2}\!}

\newcounter{statement}
\newenvironment{statement}[2][!]{%
	\vskip3mm
	\hrule
	\hrule
	\hrule
	\vskip1mm
	\noindent%
	\refstepcounter{statement}%
	\bf#2~\thestatement%
	\ifthenelse{\equal{#1}{!}}{.\ }{~(#1).\ }%
	\it%
}{%
	\vskip1mm
	\hrule
	\hrule
	\hrule
	\vskip2mm
}

\newenvironment{theorem}[1][!]{\begin{statement}[#1]{Theorem}}{\end{statement}}
\newenvironment{lemma}[1][!]{\begin{statement}[#1]{Lemma}}{\end{statement}}
\newenvironment{proposition}[1][!]{\begin{statement}[#1]{Proposition}}{\end{statement}}

\newenvironment{remark}[1][!]{\begin{statement}[#1]{Remark}}{\end{statement}}
\newenvironment{algorithm}[1][!]{\begin{statement}[#1]{Algorithm}}{\end{statement}}

\makeatletter
\def\@seccntformat#1{\vspace*{-2mm}\newline\hspace*{4mm}%
  \protect\textup{\protect\@secnumfont
    \ifnum\pdfstrcmp{subsection}{#1}=0 \bfseries\fi
    \csname the#1\endcsname
    \protect\@secnumpunct
  }%
}
\def\paragraph{\@startsection{paragraph}{4}%
  \z@\z@{-\fontdimen2\font}%
  {\normalfont\bfseries}}
\makeatother

\makeatletter
\def\section{\@startsection{section}{1}%
\z@{.7\linespacing\@plus\linespacing}{.5\linespacing}%
{\normalsize\scshape\bfseries\centering}}
\makeatother

\makeatletter
\renewcommand{\@secnumfont}{\bfseries}
\makeatother

\newcommand{\R}{\mathbb{R}}

\newcommand{\N}{\mathbb{N}}

\newcommand{\abs}[2][]{#1|\,#2\,#1|}
\newcommand{\seminorm}[3][]{#1|{#2}#1|_{#3}}
\newcommand{\setpower}[1]{\# #1}

\renewcommand{\d}[1]{\mathrm{\, d}#1}

\let\div\relax
\DeclareMathOperator{\div}{div}
\newcommand{\jump}[1]{\left[\!\left[#1\right]\!\right]}
\newcommand{\area}[1]{\Omega(#1)}
\newcommand{\transfer}[2]{\mathcal{Q}_{#2,#1}}

\newcommand{\energy}{\mathbb{E}}
\def\coarse{H}
\def\fine{h}
\def\EE{\mathcal{E}}
\def\PP{\mathcal{P}}
\def\UU{\mathcal{U}}
\def\VV{\mathcal{V}}
\def\LL{\mathcal{L}}
\def\MM{\mathcal{M}}
\def\SS{\mathcal{S}}
\def\TT{\mathcal{T}}
\def\T{\mathbb{T}}
\def\normal{\boldsymbol{\nu}}
\def\f{\boldsymbol{f}}
\def\g{\boldsymbol{g}}

\newcommand{\refine}[2]{\ifthenelse{\equal{#2}{}}{\mathrm{refine}(#1)}{\mathrm{refine}(#1;#2)}}
\newcommand{\tail}[2]{\mathrm{tail}_{#2}(#1)}
\def\coarse{H}
\def\fine{h}
\newcommand{\est}{\eta}
\newcommand{\osc}{\mathrm{osc}}
\newcommand{\data}{\varrho}

\newcommand{\Csz}{C_{\rm sz}}
\newcommand{\Cest}{C_{\rm est}}
\newcommand{\Cerr}{C_{\rm err}}
\newcommand{\Cshapereg}{C_{\rm mesh}}
\newcommand{\Crel}{C_\mathrm{rel}}
\newcommand{\Ceff}{C_\mathrm{eff}}
\newcommand{\Cmesh}{C_\mathrm{mesh}}
\newcommand{\Clde}{C_{\rm diam}}
\newcommand{\Cmark}{C_\mathrm{mark}}
\newcommand{\Cmin}{C_\mathrm{min}}
\newcommand{\Cdrel}{C_\mathrm{drel}}
\newcommand{\Cdeff}{C_\mathrm{deff}}
\newcommand{\refstack}[3][]{#1\stackrel{\mathclap{#2}}{#3}#1}

\def\reff#1#2{\stackrel{\eqref{#1}}{#2}}

\usepackage{fancyhdr}
\advance\footskip0.4cm
\advance\textheight-0.4cm
\calclayout
\newcommand*\patchAmsMathEnvironmentForLineno[1]{%
	\expandafter\let\csname old#1\expandafter\endcsname\csname #1\endcsname
	\expandafter\let\csname oldend#1\expandafter\endcsname\csname end#1\endcsname
	\renewenvironment{#1}%
	{\linenomath\csname old#1\endcsname}%
	{\csname oldend#1\endcsname\endlinenomath}}%
\newcommand*\patchBothAmsMathEnvironmentsForLineno[1]{%
	\patchAmsMathEnvironmentForLineno{#1}%
	\patchAmsMathEnvironmentForLineno{#1*}}%
\AtBeginDocument{%
	\patchBothAmsMathEnvironmentsForLineno{equation}%
	\patchBothAmsMathEnvironmentsForLineno{align}%
	\patchBothAmsMathEnvironmentsForLineno{flalign}%
	\patchBothAmsMathEnvironmentsForLineno{alignat}%
	\patchBothAmsMathEnvironmentsForLineno{gather}%
	\patchBothAmsMathEnvironmentsForLineno{multline}%
}
\usepackage[mathlines]{lineno}

\title{Instance-optimal goal-oriented adaptivity}

\author{Michael Innerberger}
\address{TU Wien, Institute for Analysis and Scientific Computing, Wiedner Hauptstr.~8--10/E101/4, 1040 Wien, Austria}
\email{Michael.Innerberger@asc.tuwien.ac.at}

\author{Dirk Praetorius}
\address{TU Wien, Institute for Analysis and Scientific Computing, Wiedner Hauptstr.~8--10/E101/4, 1040 Wien, Austria}
\email{Dirk.Praeorius@asc.tuwien.ac.at}

\thanks{{\bf Acknowledgement:} The authors acknowledge support through the Austrian Science Fund (FWF) through the doctoral school \emph{Dissipation and dispersion in nonlinear PDEs} (grant W1245), the special research program \emph{Taming complexity in PDE systems} (grant SFB F65), and the stand-alone project \emph{Optimal adaptivity for BEM and FEM-BEM coupling} (grant P27005).}

\date{\today}

\begin{document}
	
\begin{abstract}
We consider an adaptive finite element method with arbitrary but fixed polynomial degree $p \ge 1$, where adaptivity is driven by an edge-based residual error estimator.
Based on the modified maximum criterion from~{\rm[}Diening et al, \emph{Found.\ Comput.\ Math.} 16, 2016{\rm]}, we propose a goal-oriented adaptive algorithm and prove that it is instance optimal.
More precisely, the goal-error is bounded by the product of the total errors (being the sum of energy error plus data oscillations) of the primal and the dual problem, and the proposed algorithm is instance optimal with respect to this upper bound.
Numerical experiments underline our theoretical findings.
\end{abstract}
\maketitle

\input{01_introduction.tex}

\input{02_mainresults.tex}

\input{04_afem_proof.tex}

\input{05_goafem_proof.tex}

\input{06_numerics.tex}

\bibliographystyle{alpha}
\bibliography{literature}


\end{document}

%% file: 01_introduction.tex

\section{Introduction}
\label{sec:intro}

\subsection{Rate optimality vs.\ instance optimality of AFEM}\label{section:intro:versus}
For an elliptic PDE with sought solution $u \in H^1_0(\Omega)$, 
the adaptive finite element method (AFEM) iterates the loop
\begin{align}\label{eq:intro:semr}
 \boxed{\textsl{\sl\,\,SOLVE\,\,}}
 \quad \longrightarrow \quad
 \boxed{\textsl{\,\,ESTIMATE\,\,}}
 \quad \longrightarrow \quad
 \boxed{\textsl{\sl\,\,MARK\,\,}}
 \quad \longrightarrow \quad
 \boxed{\textsl{\sl\,\,REFINE\,\,}}
\end{align}
to successively adapt an initial mesh $\TT_0$ to the possible singularities of the sought solution $u$. This leads to a sequence of meshes $\TT_\ell$ and corresponding discrete solutions $u_\ell \in \SS^p_0(\TT_\ell) \subseteq H_0^1(\Omega)$ being $\TT_\ell$-piecewise polynomials of degree $p \geq 1$.
In the last decades, the mathematical understanding of AFEM has matured. We refer to~\cite{doerfler1996,mns2000,bdd04,stevenson2007,ckns2008,ffp2014} for some milestones of the analysis of convergence with optimal algebraic rates, to~\cite{msv2008,siebert2011} for a general theory on (plain) convergence, and to~\cite{axioms} for an abstract approach for optimal convergence rates. 

The works~\cite{stevenson2007,ckns2008,ffp2014,axioms} aim for rate optimality of adaptive algorithms, i.e., they prove that the usual adaptive loop~\eqref{eq:intro:semr}, based on (quasi-) minimal D\"orfler marking~\cite{doerfler1996} for the step {\sl MARK}, leads to optimal decay of the \emph{total error}
\begin{align}\label{eq:intro:total_error}
 {\rm error}(\TT_\coarse) := \enorm{u - u_\coarse} + \osc_\coarse(\TT_\coarse),
\end{align}
being the sum of energy error $\enorm{u - u_\coarse}$ plus data oscillation terms (or, equivalently, the error estimator, since $\eta_\coarse \simeq {\rm error}(\TT_\coarse)$). Let $\#(\cdot)$ denote the number of elements of a finite set. Then, rate optimality reads as follows: For all $s > 0$, it holds that
\begin{align}\label{eq:intro:rate}
 \norm{u}{\mathbb{A}_s} < \infty
 \quad \Longrightarrow \quad
 \exists \, C > 0 \, 
 \forall \, \ell \in \N_0 : \quad 
 {\rm error}(\TT_\ell) \le C \, (\#\TT_\ell)^{-s},
\end{align}
where
\begin{equation*}
	\norm{u}{\mathbb{A}_s}
	=
	\sup\limits_{\varepsilon > 0} \min\limits_{\TT_\coarse \in \T_\varepsilon}
	\big[ \varepsilon \, (\#\TT_\coarse - \#\TT_0)^s \big]
	< \infty
	\quad \text{with} \quad
	\TT_\varepsilon
	:=
	\set{\TT_\coarse \in \T}{\mathrm{error}(\TT_\coarse) \leq \varepsilon}
\end{equation*}%
corresponds to certain approximation properties of $u$ (which can be characterized in terms of Besov regularity~\cite{bddp2002,gm2014,gantumur2017}).

Unlike this, the first work~\cite{bdd04} on AFEM with convergence rates implicitly proved even instance optimality of AFEM, i.e., the total error on an adaptive mesh is quasi-optimal with respect to all refinements of $\TT_0$, which have essentially the same number of elements: 
It holds that
\begin{align}\label{eq:intro:instance}
\begin{split}
 \exists \, C > 1 \, 
 \forall \, \ell \in \N_0 \, 
 \forall \, \TT_\coarse \in {\rm refine}(\TT_0) : \quad
 &\big[ C \, \#\TT_\coarse \le \#\TT_\ell \, \Longrightarrow \,
 {\rm error}(\TT_\ell)
 \le C \, {\rm error}(\TT_\coarse) \big].
\end{split}
\end{align}
The key argument for the proof of~\eqref{eq:intro:instance} in~\cite{bdd04} is an additional coarsening step in the adaptive loop~\eqref{eq:intro:semr}. While~\cite{bdd04} also employed the D\"orfler marking criterion~\cite{doerfler1996} for the step {\sl MARK}, the work~\cite{dks16} proposed a modified maximum criterion to single out edges for refinement. For P1-AFEM for the 2D Poisson problem,~\cite{dks16} then proved instance optimality~\eqref{eq:intro:instance} of their adaptive strategy without resorting to an additional coarsening step. In~\cite{ks2016}, their analysis was extended to AFEM with non-conforming P1 elements for the Poisson problem and the Stokes system in 2D. We stress that any instance-optimal AFEM~\eqref{eq:intro:instance} is, in particular, rate optimal~\eqref{eq:intro:rate}.

\subsection{Goal-oriented adaptivity}\label{section:intro:goafem}
While standard adaptivity aims to approximate the PDE solution $u \in H^1_0(\Omega)$ by some discrete approximation $u_\ell \in \SS^p_0(\TT_\ell)$ in the energy norm, a goal-oriented adaptive finite element method (GOAFEM)  aims only to approximate $G(u)$ by $G(u_\ell)$, where $G : H^1_0(\Omega) \to \R$ is the so-called \emph{goal functional} or \emph{quantity of interest}.

In the present paper, we consider the following problem: Let $\Omega \subset \R^2$ be a polygonal Lipschitz domain, which is resolved by the initial mesh $\TT_0$, where $\TT_0$ is admissible in the sense of \cite{bdd04, nvb}. Given $f, g \in L^2(\Omega)$ and $\f, \g \in [L^2(\Omega)]^2$, the (linear) goal functional $G \in H^{-1}(\Omega)$ reads
\begin{align}\label{eq:intro:goal}
 G(u) := \int_\Omega g u - \g \cdot \nabla u \d{x},
\end{align}
where $u \in H^1_0(\Omega)$ is the unique solution to 
\begin{subequations}\label{eq:intro:primal}
\begin{align}
 \hspace*{20mm} -\div(A\nabla u) &= f + \div \f &&\hspace*{-25mm} \textrm{ in } \Omega, \hspace*{10mm} \\
		u &= 0  &&\hspace*{-25mm} \textrm{ on } \Gamma \coloneqq \partial \Omega. \hspace*{10mm}
\end{align}
\end{subequations}
For technical reasons, we assume that the diffusion matrix $A \in [L^\infty(\Omega)]^{2 \times 2}$ is $\TT_0$-piecewise constant and $A|_T \in \R^{2 \times 2}$ is symmetric and positive definite. Moreover, we assume that the restrictions $\f|_T, \g|_T$ are smooth for all $T \in \TT_0$.

Convergence and rate-optimality of GOAFEM has been addressed in~\cite{ms2008,bet2011,hpz2015,hp2016,pointabem,goafem}. The key idea of the argument is to let $u^\ast \in H^1_0(\Omega)$ be the unique solution to the \emph{dual problem}
\begin{subequations}\label{eq:intro:dual}
\begin{align}
 \hspace*{20mm} -\div(A\nabla u^\ast) &= g + \div \g &&\hspace*{-25mm} \textrm{ in } \Omega, \hspace*{25mm} \\
		u^\ast &= 0  &&\hspace*{-25mm} \textrm{ on } \Gamma. \hspace*{10mm}
\end{align}
\end{subequations}
Throughout, quantities associated with the dual problem are indexed by an asterisk. We note that the (symmetric) primal problem~\eqref{eq:intro:primal} and the dual problem~\eqref{eq:intro:dual} coincide up to the right-hand side.
Let $\enorm{v} := \norm{A^{1/2} \nabla v}{L^2(\Omega)}$ be the problem-induced energy norm.
For FEM approximations $u \approx u_\ell \in \SS^p_0(\TT_\ell)$ and $u^\ast \approx u_\ell^\ast \in \SS^p_0(\TT_\ell)$, standard duality arguments (together with the Galerkin orthogonality) lead to
\begin{align}\label{eq:intro:goafem}
\begin{split}
 |G(u) - G(u_\ell)| 
 &\le \enorm{u - u_\ell} \, \enorm{u^\ast - u^\ast_\ell}
 \\&
 \le \big[ \enorm{u - u_\ell} + \osc_\ell(\TT_\ell) \big] \, \big[ \enorm{u^\ast - u^\ast_\ell} + \osc_\ell^\ast(\TT_\ell) \big];
\end{split}
\end{align}
see, e.g.,~\cite{ms2008,goafem}.
Therefore, GOAFEM aims to control and steer the product of the total errors
\begin{align}\label{eq:intro:goafem_error}
 {\rm error}(\TT_\coarse)\,{\rm error}^\ast(\TT_\coarse) := \big[ \enorm{u - u_\coarse} + \osc_\coarse(\TT_\coarse) \big] \, \big[ \enorm{u^\ast - u^\ast_\coarse} + \osc_\coarse^\ast(\TT_\coarse) \big].
\end{align}
 While~\cite{bet2011,hpz2015,hp2016} focus on linear convergence of GOAFEM, the works~\cite{ms2008,pointabem,goafem} also prove rate optimality. All works employ variants of the D\"orfler marking criterion~\cite{doerfler1996}: The seminal work~\cite{ms2008} employs (quasi-) minimal D\"orfler marking separately for the primal and the dual problem, and then uses the smaller set for {\sl MARK}. Instead,~\cite{bet2011} proposes a (quasi-) minimal D\"orfler marking for some combined estimator. Both strategies guarantee rate optimality for the product of the total errors
\begin{align}\label{eq:intro:rate:goafem}
 \norm{u}{\mathbb{A}_s} + \norm{u^\ast}{\mathbb{A}_t} < \infty
 \,\, \Longrightarrow \,\,
 \exists \, C > 0 \, 
 \forall \, \ell \in \N_0 : \,\, 
 {\rm error}(\TT_\ell) \, {\rm error}^\ast(\TT_\ell) \le C \, (\#\TT_\ell)^{-(s+t)},
\end{align}%
where the possible algebraic rate $s + t$ now depends on the approximability properties 
of the primal and dual problem; see~\cite{ms2008,pointabem,goafem}.
While such a strategy thus leads to optimal rates for the error product, one has to admit that the latter may dramatically overestimate the actual goal error.

\subsection{Instance-optimal GOAFEM}\label{section:intro:goafem:instance}
The new GOAFEM algorithm can briefly be outlined as follows: {\sl SOLVE} computes the FEM solution $u_\ell \in \SS^p_0(\TT_\ell)$ to the primal problem~\eqref{eq:intro:primal} and $u_\ell^\ast \in \SS^p_0(\TT_\ell)$ to the dual problem~\eqref{eq:intro:dual}. {\sl ESTIMATE} computes the corresponding residual error estimators $\eta_\ell$ and $\eta_\ell^\ast$. {\sl MARK} employs the modified maximum strategy from~\cite{dks16} to obtain two sets of marked edges, namely $\overline\MM_\ell$ with respect to $\eta_\ell$ and $\overline\MM_\ell^\ast$ with respect to $\eta_\ell^\ast$. With $n := \min\{\#\overline\MM_\ell, \#\overline\MM_\ell^\ast\}$, we then define $\MM_\ell := \underline\MM_\ell \cup \underline\MM_\ell^\ast$, where $\underline\MM_\ell \subseteq \overline\MM_\ell$ and $\underline\MM_\ell^\ast \subseteq \overline\MM_\ell^\ast$ are arbitrary up to $\#\underline\MM_\ell = n = \#\underline\MM_\ell^\ast$. Finally, {\sl REFINE} employs 2D newest vertex bisection (NVB) to generate the coarsest mesh $\TT_{\ell + 1}$, where all edges in $\MM_\ell$ have been bisected once. 

The main result of the present work states that the proposed GOAFEM is instance optimal~\eqref{eq:intro:instance} with respect to the total-error product, i.e.,
\begin{align}
\label{eq:intro:instance:goafem}
 &\exists \, C > 1 \, 
 \forall \, \ell \in \N_0 \, 
 \forall \, \TT_\coarse, \TT_{\coarse^\ast} \in {\rm refine}(\TT_0) : \quad
 \\
\nonumber
 \mathclap{\hspace{-13pt}\big[ \, C} &\, \max\{ \#\TT_\coarse, \#\TT_{\coarse^\ast} \} \le \#\TT_\ell \,\, \Longrightarrow \,\,
 {\rm error}(\TT_\ell) \, {\rm error}^\ast(\TT_\ell)
 \le C \, {\rm error}(\TT_\coarse) \, {\rm error}^\ast(\TT_{\coarse^\ast}) \, \big].\hspace*{-1mm}
\end{align}%
Again, we note that this implies, in particular, rate optimality~\eqref{eq:intro:rate}.
On a technical side, we note that the seminal work~\cite{dks16} is restricted to the lowest-order FEM discretization $p = 1$, while the present analysis also allows higher (but fixed) polynomial degrees $p \ge 1$. In this sense, the present work provides also the technical tools to generalize the instance-optimal AFEM of~\cite{dks16} from $p = 1$ to general but fixed $p \ge 1$.

\subsection{Outline}
The remainder of this paper is organized as follows: In Section~\ref{sec:result}, we give a precise formulation of the modules {\sl SOLVE}, {\sl ESTIMATE}, {\sl MARK}, and {\sl REFINE} of the adaptive loop~\eqref{eq:intro:semr}. In particular, we state the modified maximum criterion (Algorithm~\ref{alg:marking}) from~\cite{dks16} as well as our extension to GOAFEM (Algorithm~\ref{alg:goafemmarking}). Then, we thoroughly formulate our GOAFEM algorithm (Algorithm~\ref{alg:abstractgoafem}) and state our main result that the proposed GOAFEM is instance optimal (Theorem~\ref{theorem:goafem}). 
Section~\ref{sec:auxiliary} collects the technical results to generalize the seminal work~\cite{dks16} from lowest-order FEM $p = 1$ to arbitrary polynomial degree $p \ge 1$ (Theorem~\ref{theorem:instanceoptimality}) by reviewing the proof of~\cite{dks16} in an abstract framework.
Thereafter, Section~\ref{sec:proofgoafem} gives the proof of Theorem~\ref{theorem:goafem}. 
Some numerical experiments in Section~\ref{sec:numerics} conclude this work and empirically compare the instance-optimal GOAFEM algorithm from the present work with the rate-optimal GOAFEM strategies from~\cite{ms2008,bet2011,goafem}.

\subsection{General notation}
In all results, the involved constants (as well as their dependencies) are stated explicitly.
In proofs, however, we write $\lesssim$ to abbreviate $\le$ up to a multiplicative constant which is clear from the context.
Moreover, we write $\simeq$ if both estimates, $\lesssim$ and $\gtrsim$, hold.

%% file: 02_mainresults.tex

\section{Main result}
\label{sec:result}

\noindent
Before stating our main result, we discuss the particular modules of the adaptive loop~\eqref{eq:intro:semr} and fix the necessary notation.

\subsection{\textsl{REFINE}}
A mesh $\TT_\coarse$ is a conforming triangulation of $\Omega$ into non-degenerate compact triangles $T \in \TT_\coarse$. The edges of $\TT_\coarse$ are denoted by $\EE_\coarse$. The set of interior edges of $\TT_\coarse$ is denoted by $\EE_\coarse^{\Omega}$, i.e., for each $E \in \EE_\coarse^{\Omega}$, there exist unique $T, T' \in \TT_\coarse$ such that $E = T \cap T'$. The set of vertices of $\TT_\coarse$ is denoted by $\VV_\coarse$. We define the patches
\begin{align}
 \TT_\coarse(\omega) : = \set{T \in \TT_\coarse}{T \cap \omega \neq \emptyset}
 \quad \text{for all } \omega \subset \overline\Omega.
\end{align}
For vertices $z \in \VV_\coarse$, we abbreviate $\TT_\coarse(z) := \TT_\coarse(\{z\}) = \set{T \in \TT_\coarse}{z \in T}$. For neighbors $T, T'\in \TT_\coarse$, we also consider the reduced edge patch
\begin{align}
 \TT_\coarse^{\rm red}(E) := \set{T \in \TT_\coarse}{E \subset \partial T} = \{T, T'\}
 \quad \text{for } E = T \cap T' \in \EE_\coarse.
\end{align}
Similarly, we define the area associated to a set of triangles $\UU_\coarse \subseteq \TT_\coarse$ by
\begin{align}
 \Omega(\UU_\coarse) := \bigcup_{T \in \UU_\coarse} T \subseteq \overline\Omega
 \quad \text{for all } \UU_\coarse \subseteq \TT_\coarse.
\end{align}

For mesh-refinement, we employ an edge-based variant of newest vertex bisection (NVB) \cite{nvb}. We suppose that the initial mesh $\TT_0$ is admissible in the sense of~\cite{bdd04,nvb}: For all neighbors $T, T' \in \TT_0$, the joint edge $E = T \cap T' \in \EE_0$ is the reference edge of $T$ if and only if it is also the reference edge of $T'$. While this assumption is unnecessary for the NVB algorithm~\cite{kpp2013}, it provides additional structure which is crucial in the instance-optimality analysis of~\cite{dks16}.

For a mesh $\TT_\coarse$ and a set $\MM_\coarse \subseteq \EE_\coarse$ of marked edges, let $\TT_\fine := {\rm refine}(\TT_\coarse, \MM_\coarse)$ be the coarsest NVB refinement of $\TT_\coarse$ such that all edges $E \in \MM_\coarse$ have been bisected. Moreover, we write $\TT_\fine \in {\rm refine}(\TT_\coarse)$, if $\TT_\fine$ can be obtained by finitely many steps of NVB refinement. Then, $\T := {\rm refine}(\TT_0)$ is the set of all possible NVB refinements of $\TT_0$.
We note that NVB leads to uniformly shape-regular meshes in the sense of
\begin{equation}\label{eq:shapereg}
	\Cshapereg := \sup_{\TT_\coarse \in \T}
	\max_{T \in \TT_\coarse}
	\frac{\mathrm{diam}(T)^2}{|T|}
	< \infty,
\end{equation}
where $|T|$ is the area of a triangle $T$.

\subsection{\textsl{SOLVE}}
As usual, the primal problem~\eqref{eq:intro:primal} is understood in weak form. The Lax--Milgram lemma guarantees existence and uniqueness of $u \in H^1_0(\Omega)$ such that
\begin{equation}\label{eq:weakform}
	a(u,v) := \int_{\Omega} A \nabla u \cdot \nabla v \d{x}
	=
	\int_{\Omega} f v - \f \cdot \nabla v \d{x} =: F(v)
	\quad \text{for all }
	v \in H^1_0(\Omega).
\end{equation}
We define the energy norm $\enorm[]{v} \coloneqq \norm{A^{1/2} \nabla v}{L^2(\Omega)} = a(v,v)^{1/2}$ and note that $\enorm{\cdot} \simeq \seminorm{\cdot}{H^1(\Omega)} \coloneqq \norm[]{\nabla (\cdot)}{L^2(\Omega)} \simeq \norm{\cdot}{H^1(\Omega)}$ on $H^1_0(\Omega)$.
Let $\TT_\coarse \in \T$ and $p \ge 1$.  For the discretization of~\eqref{eq:weakform}, define the space of $\TT_\coarse$-piecewise polynomials
\begin{align}
 \PP^p(\TT_\coarse) := \set{v \in L^2(\Omega)}{\forall T \in \TT_\coarse \quad v|_T \text{ is a polynomial of degree } \le p}
\end{align}
as well as the conforming FEM spaces
\begin{align}
 \SS^p(\TT_\coarse) \coloneqq \PP^p(\TT_\coarse) \cap C(\Omega) = \PP^p(\TT_\coarse) \cap H^1(\Omega)
 \quad \text{and} \quad
 \SS^p_0(\TT_\coarse) \coloneqq \mathcal{S}^p(\TT_\coarse) \cap H^1_0(\Omega).
\end{align}
Again, the Lax--Milgram lemma proves the existence and uniqueness of $u_\coarse \in \SS^p_0(\TT_\coarse)$ such that
\begin{equation}\label{eq:weakfemform}
	a(u_\coarse,v_\coarse) = F(v_\coarse)
	\quad \text{for all }
	v_\coarse \in \SS^p_0(\TT_\coarse).
\end{equation}

\subsection{\textsl{ESTIMATE}}
For {\sl a~posteriori} error estimation, we employ an edge-based residual error estimator.
Let $|E|$ be the length of an edge $E$.
For the primal problem~\eqref{eq:intro:primal} with $F = f + \div\f \in H^{-1}(\Omega)$, we define
\begin{align}
\label{eq:def-eta}
 \begin{split}
 \eta_\coarse(E)^2 
 &\coloneqq |E| \, \norm{\jump{(A \nabla u_\coarse + \f)\cdot\normal}}{L^2(E)}^2
 \\&
 \qquad + \sum_{T \in \TT_\coarse^{\rm red}(E)} |T| \, \norm{f + \div (A \nabla u_\coarse + \f)}{L^2(T)}^2
 \quad \text{for all } E \in \EE_\coarse,
 \end{split}
\end{align}
where $\normal$ is a normal vector on $E$ and $\jump{\cdot}$ denotes the jump across $E$.
With this notation, we define 
\begin{equation}
\label{eq:def-eta-subset}
 \eta_\coarse(\UU_\coarse) := \bigg( \sum_{E \in \UU_\coarse} \eta_\coarse(E)^2 \bigg)^{1/2}.
\end{equation}
With the $L^2$-orthogonal projections $\Pi_E : L^2(E) \to \PP^{p-1}(E)$ and $\Pi_T : L^2(T) \to \PP^{p-2}(T)$, where $\PP^{-1}(T) := \{0\}$, the data resolution terms read
\begin{equation}
\label{eq:def-oscillations}
	 \data_\coarse(T)^2
	 \coloneqq |T| \, \norm{(1-\Pi_T)(f + \div \f)}{L^2(T)}^2
	 +\!\! \sum_{\substack{E \in \EE_\coarse\\E \subset \partial T}} |T|^{1/2} \, \norm{(1-\Pi_E)\jump{\f \cdot \normal}}{L^2(E)}^2.
\end{equation}
Note that for $p=1$ the volume term of $\data_\coarse(T)$ simply reads $|T| \, \norm{f + \div \f}{L^2(T)}^2$.
Finally, with the $L^2$-orthogonal projection $\Pi_T^+ : L^2(T) \to \PP^{p-1}(T)$, the so-called data oscillations read
\begin{equation}\label{eq:def-classical-oscillations}
	\osc_\coarse(T)^2
	\coloneqq |T| \, \norm{(1-\Pi_T^+)(f + \div \f)}{L^2(T)}^2
	+\!\! \sum_{\substack{E \in \EE_\coarse\\E \subset \partial T}} |T|^{1/2} \, \norm{(1-\Pi_E)\jump{\f \cdot \normal}}{L^2(E)}^2.
\end{equation}
We note that
\begin{align}\label{eq:osc-eta}
 \osc_\coarse(T) \leq
 \data_\coarse(T) \lesssim \eta_\coarse(E)
 \qquad \text{for all } T \in \TT_\coarse, \, E \in \EE_\coarse
 \text{ with } E \subseteq \partial T,
\end{align}
where the hidden constant depends only on $\Cshapereg$ from~\eqref{eq:shapereg}.
For a subset $\UU_\coarse \subseteq \TT_\coarse$, we define $\data_\coarse(\UU_\coarse)$ and $\osc_\coarse(\UU_\coarse)$ analogously to~\eqref{eq:def-eta-subset}.
We note that
\begin{equation}
\label{eq:reliability}
 \Crel^{-1} \, \enorm{\! u - u_\coarse \!} 
 \le \eta_\coarse(\EE_\coarse)
 \le \Ceff \, \big[ \enorm{\! u - u_\coarse \!} + \osc_\coarse(\TT_\coarse) \big]
 \le \Ceff \, \big[ \enorm{\! u - u_\coarse \!} + \data_\coarse(\TT_\coarse) \big], 
\end{equation}%
where the reliability constant $\Crel > 0$ depends only on $\Cshapereg$ from~\eqref{eq:shapereg}, while the efficiency constant $\Ceff > 0$ depends additionally on $p$.
In general, the data resolution term $\data_\coarse$ cannot be expected to be of higher order than $\textrm{error}(\TT_\coarse)$.
However, from~\eqref{eq:osc-eta}--\eqref{eq:reliability} one infers that there holds
\begin{equation}
\label{eq:data-osc-equivalence}
	\textrm{error}(\TT_\coarse) + \osc_\coarse(\TT_\coarse)
	\simeq
	\textrm{error}(\TT_\coarse) + \data_\coarse(\TT_\coarse).
\end{equation}

\subsection{\textsl{MARK}}
\label{subsec:marking}
Let $\TT_\coarse \in \T$. We define the tail of an edge $E \in \EE_\coarse$ by
\begin{align}
 \tail{E}{\coarse} \coloneqq \EE_\coarse \setminus \EE_{\coarse,E},
 \quad \text{where} \quad
 \TT_{\coarse,E} := {\rm refine}(\TT_\coarse,\{E\}),
\end{align}
i.e., the tail consists of all edges, which have to be refined to ensure conformity of the triangulation if $E$ is bisected. Moreover, we define
\begin{align}
 \tail{\UU_\coarse}{\coarse}
 := \bigcup_{E \in \UU_\coarse} \tail{E}{\coarse}
 \quad \text{for all } \UU_\coarse \subseteq \EE_\coarse.
\end{align}
With these definitions, we recall the following modified maximum criterion from~\cite[Algorithm~5.1 Step~(3)]{dks16}, which leads to an instance-optimal AFEM.

\begin{algorithm}[Modified maximum criterion]\label{alg:marking}\\
	\textbf{Input:} Edges $\EE_\coarse$ with indicators $\mu_\coarse := \big(\mu_\coarse(E)\big)_{E \in \EE_\coarse}$, marking parameter $0 < \vartheta \leq 1$.\\
	\textbf{Output:} Set $\MM_\coarse := {\rm markAFEM}(\EE_\coarse,\mu_\coarse,\vartheta) \subseteq \EE_\coarse$ of marked edges.
	\begin{algorithmic}[1]
		\STATE $\MM_\coarse \coloneqq \emptyset$ and $\UU \coloneqq \EE_\coarse$
		\STATE $M \coloneqq \max \set{\mu_\coarse(\tail{E}{\coarse})}{E \in \EE_\coarse}$
		\WHILE{$\UU \neq \emptyset$}
		\STATE pick $E \in \UU$ and update $\UU \coloneqq \UU \setminus \tail{E}{\coarse}$
		\STATE compute $m \coloneqq \eta_\coarse(\tail{E}{\coarse} \setminus \tail{\MM_\coarse}{\coarse})$
		\IF{$m \geq \vartheta M$}
		\STATE $\MM_\coarse \coloneqq \MM_\coarse \cup \{E\}$
		\ENDIF
		\ENDWHILE
	\end{algorithmic}
\end{algorithm}

\def\Z{\mathbb{Z}}
We refer to~\cite[Algorithm~7.2]{dks16} for a recursive implementation of Algorithm~\ref{alg:marking}, which has linear costs.

\subsection{Instance-optimal AFEM}
\label{subsec:instanceoptimality}
The work~\cite{dks16} analyzes the following instance of the adaptive loop~\eqref{eq:intro:semr}, which turns out to be instance-optimal; see Theorem~\ref{theorem:instanceoptimality}.

\begin{algorithm}[Instance-optimal AFEM]\label{alg:abstractafem}\\
\textbf{Input:} Initial mesh $\TT_0$, polynomial degree $p \in \N$, marking parameter $0 < \vartheta \leq 1$.\\
\textbf{Output:} Meshes $\TT_\ell$, discrete solutions $u_\ell$, and estimators $\eta_\ell(\EE_\ell)$ for all $\ell \in \N_0$.
\begin{algorithmic}[1]
\FOR{{\bfseries all} $\ell = 0,1,2, \dots$}
\STATE compute FEM solution $u_\ell \in \SS^p_0(\TT_\ell)$
\STATE compute indicators $\eta_\ell = \big( \eta_\ell(E) \big)_{E \in \EE_\ell}$
\STATE generate $\MM_\ell := {\rm markAFEM}(\EE_\ell, \eta_\ell, \vartheta)$ by Algorithm~\ref{alg:marking}
\STATE employ NVB to generate $\TT_{\ell+1} = {\rm refine}(\TT_\ell,\MM_\ell)$
\ENDFOR
\end{algorithmic}
\end{algorithm}

For $p = 1$, the following theorem is the main result of~\cite{dks16}. Our analysis below implies that the result remains true for arbitrary polynomial degrees $p \ge 1$. 

\begin{theorem}\label{theorem:instanceoptimality}
Let the initial mesh $\TT_0$ be admissible in the sense of~\cite{bdd04}. Let $p \in \N$ and $0 < \vartheta \leq 1$. 
Then, the AFEM Algorithm~\ref{alg:abstractafem} for the primal problem~\eqref{eq:intro:primal} is instance optimal with respect to the total error, i.e.,
\begin{align}
\label{eq:pfemopt}
	&\exists \, C > 1 \, 
	\forall \, \ell \in \N_0 \, 
	\forall \, \TT_\coarse \in {\rm refine}(\TT_0) : \\
	\nonumber
	\mathclap{\hspace{-28pt}\Big(  C \, \#}&(\TT_\coarse \backslash \TT_0) \le \#(\TT_\ell \backslash \TT_0) \,\,\Longrightarrow \,\,
	\enorm{u - u_\ell}^2 + \osc_\ell(\TT_\ell)^2
	\le C \big[ \enorm{u - u_\coarse}^2 + \osc_\coarse(\TT_\coarse)^2\big] \Big).
\end{align}%
The constant $C$ depends only on $\vartheta$, $p$, $\Cshapereg$, and the data $A$, $f$, $\f$.
\end{theorem}

\begin{remark}
We note that elementary calculation shows that, for all $\TT_\coarse \in \T \backslash \{\TT_0\}$,
\begin{align*}
 \#\TT_\coarse - \#\TT_0  
 \le \#(\TT_\coarse \backslash \TT_0) 
 \le \#\TT_\coarse
 \le (\#\TT_0) \, ( \#\TT_\coarse - \#\TT_0 + 1 )
 \le 2 \, (\#\TT_0) \, ( \#\TT_\coarse - \#\TT_0 );
\end{align*}
see, e.g.,~\cite[Lemma~22]{bhp2018}. Hence, $\#(\TT_\ell \backslash \TT_0)$ in~\eqref{eq:pfemopt} can, in fact, be replaced by $\#\TT_\ell$ (at the cost that the constant $C$ in~\eqref{eq:intro:instance} will additionally depend on $\#\TT_0$).
Therefore, the statement of Theorem~\ref{theorem:instanceoptimality} is equivalent to the introductory statement of instance optimality~\eqref{eq:intro:instance} in Section~\ref{section:intro:versus}.
\end{remark}

\subsection{Instance-optimal GOAFEM}

For GOAFEM, additionally the dual problem~\eqref{eq:intro:dual} is solved analogously to~\eqref{eq:weakform} and~\eqref{eq:weakfemform} with $F(\cdot)$ being replaced by $G(\cdot)$ from~\eqref{eq:intro:goal}.
The Lax--Milgram lemma guarantees existence and uniqueness of the dual solution $u^\ast \in H^1_0(\Omega)$ and its FEM approximation $u_\coarse^\ast \in \SS^p_0(\TT_\coarse)$.
For the dual problem~\eqref{eq:intro:dual} with $G = g + \div \g \in H^{-1}(\Omega)$, we define the \emph{dual error estimator}
\begin{align*}
\begin{split}
	\eta^\ast_\coarse(E)^2 
	&\coloneqq |E| \, \norm{\jump{(A \nabla u_\coarse^\ast + \g)\cdot\normal}}{L^2(E)}^2
	\\&
	\qquad + \sum_{T \in \TT_\coarse^{\rm red}(E)} |T| \, \norm{g + \div (A \nabla u_\coarse^\ast + \g)}{L^2(T)}^2
	\quad \text{for all } E \in \EE_\coarse,
\end{split}
\end{align*}
analogously to~\eqref{eq:def-eta} and $\eta_\coarse^\ast(\UU_\coarse)$ is understood as in~\eqref{eq:def-eta-subset}.
The data resolution terms $\data^\ast_\coarse(T)$ (and oscillations $\osc^\ast_\coarse(T)$) for the dual problem are defined analogously to~\eqref{eq:def-oscillations} (and~\eqref{eq:def-classical-oscillations}) with $g$ and $\g$ instead of $f$ and $\f$, respectively.
Again, we note that
\begin{align}
\label{eq:osc-eta-dual}
	\osc_\coarse^\ast(T) \leq
	\data_\coarse^\ast(T) \lesssim \eta_\coarse^\ast(E)
	\quad\text{for all } T \in \TT_\coarse, \, E \in \EE_\coarse
	\text{ with } E \subseteq \partial T,
\end{align}
where the hidden constant depends only on $\Cshapereg$ from~\eqref{eq:shapereg}.
For a subset $\UU_\coarse \subseteq \TT_\coarse$, we define $\data^\ast_\coarse(\UU_\coarse)$ and $\osc^\ast_\coarse(\UU_\coarse)$ analogously to~\eqref{eq:def-eta-subset}.
We note that
\begin{equation}
\begin{split}
\label{eq:reliability-dual}
	\Crel^{-1} \, \enorm{ u^\ast-  u^\ast_\coarse } 
	&\le \eta^\ast_\coarse(\EE_\coarse) \\
	&\le \Ceff \big[ \enorm{ u^\ast - u^\ast_\coarse } + \osc^\ast_\coarse(\TT_\coarse) \big]
	\le \Ceff \big[ \enorm{ u^\ast - u^\ast_\coarse } + \data^\ast_\coarse(\TT_\coarse) \big],
\end{split}
\end{equation}%
with the same constants $\Crel$ and $\Ceff$ as in~\eqref{eq:reliability}.

In case of GOAFEM, the estimate~\eqref{eq:intro:goafem} essentially reduces the problem to simultaneously solving two independent linear problems.
Therefore, using the AFEM Algorithm~\ref{alg:abstractafem} for the primal and dual problem separately and, in each step, taking the overlay of refinements is easily observed to be instance optimal when the refined sets have comparable cardinality.
This can be achieved by not refining some of the edges in the larger set of marked edges, since Proposition~\ref{prop:marking}~below guarantees that instance optimality does not depend on the number of marked edges per step.

The preceeding strategy is realized by our following marking strategy (Algorithm~\ref{alg:goafemmarking}).
For the formal statement, recall the Gauss brackets $\lfloor x \rfloor := \max \set{n \in \Z}{n \le x}$ for $x \in \R$.
We note that the following algorithm is slightly more general than the strategy outlined in Section~\ref{section:intro:goafem:instance} of the introduction (where $\Cmin = 1$).

\begin{algorithm}[Modified maximum criterion for GOAFEM]\label{alg:goafemmarking}\\
	\textbf{Input:} Edges $\EE_\coarse$, indicators $\eta_\coarse := \big(\eta_\coarse(E)\big)_{E\in \EE_\coarse}$ and $\eta^\ast_\coarse := \big(\eta^\ast_\coarse(E)\big)_{E\in \EE_\coarse}$, marking parameters $0 < \vartheta \leq 1$ and $\Cmin > 0$.\\
	\textbf{Output:} Set $\MM_\coarse := {\rm markGOAFEM}(\EE_\coarse,\eta_\coarse,\eta^\ast_\coarse,\vartheta,\Cmin) \subseteq \EE_\coarse$ of marked edges.
	\begin{algorithmic}[1]
		\STATE generate $\overline\MM_\coarse := {\rm markAFEM}(\EE_\coarse, \eta_\coarse, \vartheta)$ by Algorithm~\ref{alg:marking}
		\STATE generate $\overline\MM^\ast_\coarse := {\rm markAFEM}(\EE_\coarse, \eta^\ast_\coarse, \vartheta)$ by Algorithm~\ref{alg:marking}
		\STATE choose $\MM_{\rm min} \coloneqq \arg\min \big\{ \#\overline\MM_\coarse \,,\, \#\overline\MM^\ast_\coarse \}$ 
		and $\MM_{\rm max} := \big\{ \overline\MM_\coarse \,,\, \overline\MM^\ast_\coarse \big\} \backslash \big\{ \MM_{\rm min} \big\}$
		\STATE define $n \coloneqq \min \big\{
		\#\MM_{\rm max}, \max \{
		1, \lfloor
		\Cmin \#\MM_{\rm min}
		\rfloor \} \big\}$
		\STATE pick $\MM_{\rm max}' \subseteq \MM_{\rm max}$ with $\#\MM_{\rm max}' = n$
		\STATE choose $\MM_\coarse = \MM_{\rm min} \cup \MM_{\rm max}'$
	\end{algorithmic}
\end{algorithm}

As outlined in the introduction, the main idea behind GOAFEM is the duality-based estimate
\begin{equation*}
	\abs{G(u) - G(u_\ell)}~
	= 
	\abs{a(u - u_\ell,u^\ast)}~
	= 
	\abs{a(u - u_\ell,u^\ast - u^\ast_\ell)}
	\leq~
	\enorm{u - u_\ell} \enorm{u^\ast - u^\ast_\ell},
\end{equation*}
The formal statement of our GOAFEM algorithm reads as follows:

\begin{algorithm}[Instance-optimal GOAFEM]\label{alg:abstractgoafem}\\
\textbf{Input:} Initial mesh $\TT_0$, polynomial degree $p \in \N$, marking parameters $0 < \vartheta \leq 1$ and $\Cmin > 0$.\\
\textbf{Output:} Meshes $\TT_\ell$, discrete solutions $u_\ell, u_\ell^\ast$, estimators $\eta_\ell(\EE_\ell), \eta_\ell^\ast(\EE_\ell)$ and goal quantities $G(u_\ell)$ for all $\ell \in \N_0$.
\begin{algorithmic}[1]
\FOR{{\bfseries all} $\ell = 0,1,2, \dots$}
\STATE compute FEM solutions $u_\ell \in \SS^p_0(\TT_\ell)$ and $u_\ell^\ast \in \SS^p_0(\TT_\ell)$
\STATE compute indicators $\eta_\ell = \big( \eta_\ell(E) \big)_{E \in \EE_\ell}$ and $\eta_\ell^\ast = \big( \eta_\ell^\ast(E) \big)_{E \in \EE_\ell}$
\STATE generate $\MM_\ell := {\rm markGOAFEM}(\EE_\ell, \eta_\ell, \eta_\ell^\ast, \vartheta, \Cmin)$ by Algorithm~\ref{alg:goafemmarking}
\STATE employ NVB to generate $\TT_{\ell+1} = {\rm refine}(\TT_\ell,\MM_\ell)$
\ENDFOR
\end{algorithmic}
\end{algorithm}

The following theorem is the main result of this work. We stress that the theorem involves the adaptively generated mesh $\TT_\ell$ for the primal and the dual error and compares it with arbitrary meshes $\TT_\coarse$ and $\TT_{\coarse^\ast}$, where $\TT_\coarse$ is used for the primal error and $\TT_{\coarse^\ast}$ is used for the dual error.

\begin{theorem}\label{theorem:goafem}
Let the initial mesh $\TT_0$ be admissible in the sense of~\cite{bdd04}. Let $p \in \N$ and $0 < \vartheta \leq 1$ as well as $\Cmin > 0$.
Let $(\TT_\ell)_{\ell \in \N_0}$ be the sequence of meshes generated by Algorithm~\ref{alg:abstractgoafem}.
Then, the AFEM Algorithm~\ref{alg:abstractafem} is instance optimal with respect to the product of total errors, i.e., \quad $\exists \, C > 1 \, 
	\forall \, \ell \in \N_0 \, 
	\forall \, \TT_\coarse, \TT_{\coarse^\ast} \in {\rm refine}(\TT_0) :$
\begin{equation}
\label{eq:goafemio}
\begin{split}
	& \Big( C \, \max \{\#(\TT_\coarse \backslash \TT_0), \#(\TT_{\coarse^\ast} \backslash \TT_0)\} \le \#(\TT_\ell \backslash \TT_0) \\
	&\qquad \qquad \Longrightarrow \,
	\big[ \, \enorm{u - u_\ell}^2 + \osc_\ell(\TT_\ell)^2 \, \big]
	\big[ \, \enorm{u^\ast - u^\ast_\ell}^2 + \osc^\ast_\ell(\TT_\ell)^2 \, \big]\\
	& \qquad \quad \qquad \qquad \leq
	C \big[ \, \enorm{u - u_\coarse}^2 + \osc_\coarse(\TT_\ell)^2 \, \big]
	\big[ \, \enorm{u^\ast - u_{\coarse^\ast}^\ast}^2 + \osc^\ast_{\coarse^\ast}(\TT_\ell)^2 \, \big] \Big).
\end{split}
\end{equation}
The constant $C$ depends only on $\vartheta$, $p$, $\Cshapereg$ $\Cmin$, and the data $A$, $f$, $\f$, $g$, $\g$.
\end{theorem}

\begin{remark}
	Note that the natural statement of instance-optimality for GOAFEM in the sense of \eqref{eq:intro:instance} and \eqref{eq:intro:goafem_error} would be: 
	$\exists \, C > 1 \, 
	\forall \, \ell \in \N_0 \, 
	\forall \, \TT_\coarse \in {\rm refine}(\TT_0)$ such that
	\begin{align*}
		&\Big( C \#(\TT_\coarse \backslash \TT_0) \le \#(\TT_\ell \backslash \TT_0) \\
		&\qquad \qquad \Longrightarrow \,
		\big[ \, \enorm{u - u_\ell}^2 + \osc_\ell(\TT_\ell)^2 \, \big]
		\big[ \, \enorm{u^\ast - u^\ast_\ell}^2 + \osc^\ast_\ell(\TT_\ell)^2 \, \big]\\
		& \qquad \qquad \qquad \qquad \quad \leq
		C \big[ \, \enorm{u - u_\coarse}^2 + \osc_\coarse(\TT_\ell)^2 \, \big]
		\big[ \, \enorm{u^\ast - u_{\coarse}^\ast}^2 + \osc^\ast_{\coarse}(\TT_\ell)^2 \, \big] \Big).
	\end{align*}
	Our Theorem~\ref{theorem:goafem}, however, is stronger.
	There, the mesh for the right-hand side can be chosen for both factors independently.
\end{remark}

%% file: 04_afem_proof.tex
\section{Auxiliary results}
\label{sec:auxiliary}

In this section, we present four properties~{\rm\ref{enum:A:mark}--\ref{enum:A:dle}} that are sufficient for instance optimality.
We further show, how they are proved for our model problem from Section~\ref{sec:result}. 
In particular, we generalize the analysis of~\cite{dks16} from lowest-order FEM $p = 1$ to arbitrary fixed order $p \ge 1$.

\subsection{Abstract result on instance optimality}
\label{subsec:abstractio}

This subsection aims to review the proof of \cite[Theorem~7.3]{dks16} in an abstract framework. 
For arbitrary $m \in \N$, the tuple $(\TT_\coarse, \TT_\fine ; \TT_1,\ldots, \TT_m) \in \T^{m+2}$ is a \emph{diamond}, if 
\begin{itemize}
	\item $\TT_j \in \T$ are meshes for all $j = 1, \dots, m$
	\item with finest common coarsening $\TT_\coarse \in \T$ and coarsest common refinement $\TT_\fine \in \T$
	\item such that the areas $\Omega(\TT_j \backslash \TT_\fine)$ are pairwise disjoint for all $j = 1, \dots, m$.
\end{itemize}
Note that $\TT_\coarse, \TT_\fine \in \T$ exist (and are unique), since newest vertex bisection is a a binary refinement rule, where the order of refinements does not matter. This allows to write
\begin{align*}
\TT_\coarse &= 
\bigcup_{j = 1}^m 
\set{T \in \TT_j}{\forall k \in \{1, \dots, m\} \, \forall T' \in \TT_k \quad \big( \, T \subseteq T' \, \Longrightarrow \, T = T' \, \big) },
\\
\TT_\fine &= 
\bigcup_{j = 1}^m 
\set{T \in \TT_j}{\forall k \in \{1, \dots, m\} \, \forall T' \in \TT_k \quad \big( \, T' \subseteq T \, \Longrightarrow \, T = T' \, \big) }.
\end{align*}
Diamonds are a means to couple the lattice structure of $\T$ with an abstract energy 
\begin{align}\label{eq:def:abstract:energy}
\energy : \T \to \R_{\ge0}.
\end{align}%
Only energies that are compatible with this structure are suitable to prove instance optimality. This is encoded in the following properties~\ref{enum:A:mark}--\ref{enum:A:equiv}, where $\Cmark$, $\Clde$, $\Cest$, $\Cerr > 0$ are generic constants, $\eta_\fine$ is a computable edge-based estimator, and {\tt mark} is an abstract marking strategy:

\begin{enumerate}[label=(A\arabic*)]
	\bf 
	\item \label{enum:A:mark}
	Marking criterion: 
	\rm
	For all meshes $\TT_\coarse \in \T$ with  edges $\EE_\coarse$, the marking strategy guarantees that the marked edges $\MM_\coarse := {\tt mark}(\EE_\coarse, (\eta_\coarse(E))_{E \in \EE_\coarse})$ satisfy that
	\begin{equation*}
	\MM_\coarse \neq \emptyset
	\quad \textrm{and} \quad
	\est_\coarse (\tail{\MM_\coarse}{\coarse})^2
	\geq
	\Cmark \, ( \setpower{\MM_\coarse} ) \,
	\max_{E \in \EE_\coarse} \est_\coarse(\tail{E}{\coarse})^2.
	\end{equation*}
	\bf	
	\item \label{enum:A:mon}
	Monotonicity of energy: 
	\rm
	For all $\TT_\coarse \in \T$ and all $\TT_\fine \in~\refine{\TT_\coarse}{}$, it holds that
	\begin{equation*}
	0 \le \energy(\TT_\fine) \leq \energy(\TT_\coarse).	
	\end{equation*}
	\bf
	\item \label{enum:A:lde}
	Diamond estimate: 
	\rm 
	For all diamonds $(\TT_\coarse, \TT_\fine; \TT_1,\ldots, \TT_m) \in \T^{m+2}$, it holds that
	\begin{equation*}
	\Clde^{-1} \, \big[ \, \energy(\TT_\coarse) - \energy(\TT_\fine) \, \big]
	\le \sum_{j=1}^{m} \big[ \, \energy(\TT_j) - \energy(\TT_\fine) \, \big]
	\le \Clde \, \big[ \, \energy(\TT_\coarse) - \energy(\TT_\fine) \, \big].
	\end{equation*}
	\bf	
	\item \label{enum:A:dle}
	Local energy estimates for the estimator:
	\rm
	For all $\TT_\coarse \! \in \! \T$ and $\TT_\fine \in\!~\!\refine{\TT_\coarse}{}$, it holds that
	\begin{equation*}
	\Cest^{-1} \big[ \, \est_\coarse(\EE_\coarse \setminus \EE_\fine)^2 \, \big]
	\le \energy(\TT_\coarse) - \energy(\TT_\fine)
	\le  \Cest \big[ \, \est_\coarse(\EE_\coarse \setminus \EE_\fine)^2 \, \big].
	\end{equation*}
	\bf
	\item \label{enum:A:equiv}
	Equivalence of energy and total error: 
	\rm
	For all $\TT_\coarse \in \T$, it holds that
	\begin{equation*}
	\Cerr^{-1} \, \energy(\TT_\coarse) 
	\le 
	\enorm{u - u_\coarse}^2 + \osc_\coarse(\TT_\coarse)^2
	\le \Cerr \, \energy(\TT_\coarse) 
	\end{equation*}
\end{enumerate}

As can be seen from the proof of \cite[Theorem~7.3]{dks16}, the conditions~\ref{enum:A:mark}--\ref{enum:A:dle} are sufficient for an AFEM to be instance optimal with respect to $\energy$. Moreover, condition~\ref{enum:A:equiv} allows to derive 
instance optimality with respect to the total error.
We formulate this as a proposition, but refer to~\cite{dks16} for the proof.

\begin{proposition}\label{theorem:abstractio}
	Consider an AFEM loop as given by \eqref{eq:intro:semr}, which satisfies the conditions {\rm\ref{enum:A:mark}--\ref{enum:A:dle}}. Then, the AFEM is instance optimal with respect to the energy, i.e.,
	\begin{align}\label{eq:dks16}
	\begin{split}
	\exists \, C \! > \! 1 \, 
	\forall \, \ell \in \N_0 \, 
	\forall \, \TT_\coarse \! \in \! {\rm refine}(\TT_0) : ~
	&\big[ C \, \#(\TT_\coarse \backslash \TT_0) \le \#(\TT_\ell \backslash \TT_0) \, \Longrightarrow \,
	\energy(\TT_\ell)
	\le \energy(\TT_\coarse) \big].
	\end{split}
	\end{align}
	If {\rm\ref{enum:A:equiv}} is satisfied in addition, then the AFEM is instance optimal in the sense of~\eqref{eq:intro:instance}.
	\qed
\end{proposition}

\begin{remark}
	We note that the proof of Proposition~\ref{theorem:abstractio} (resp.~\cite[Theorem~7.3]{dks16}) is currently tailored to 2D newest vertex bisection, for which structural properties are exploited (so-called \emph{populations}). Besides this, the proof only relies on the given properties~{\rm\ref{enum:A:mark}--\ref{enum:A:dle}}, as was already observed in \cite{ks2016}.
\end{remark}

\subsection{Verification of~\ref{enum:A:mark}: Marking criterion}\label{subsec:A1}
In \cite[Proposition~5.1]{dks16}, it is shown that Algorithm~\ref{alg:marking} satisfies the marking criterion~\ref{enum:A:mark} with $\Cmark = \vartheta$.
We state the following proposition, which is a straightforward generalization of this result and actually follows from the same arguments.

\begin{proposition}\label{prop:marking}
	Let $\overline{\MM}_\coarse \subseteq \EE_\coarse$ be the set of edges marked by Algorithm~\ref{alg:marking} for $0 < \vartheta \leq 1$.
		Then, any subset $\MM_\coarse \subseteq \overline{\MM}_\coarse$  with $\MM_\coarse \neq \emptyset$ satisfies {\rm\ref{enum:A:mark}} with $\Cmark = \vartheta$.
		\hfill $\square$
\end{proposition}

\subsection{Verification of~\ref{enum:A:mon}: Monotonicity of energy}\label{subsec:A2}

We consider the energy~\eqref{eq:def:abstract:energy} corresponding to a mesh $\TT_\coarse \in \T$ by
\begin{equation}
\label{eq:def:energy}
	\energy(\TT_\coarse)
	:= \frac{1}{2} \, a(u_\coarse,u_\coarse) - F(u_\coarse)
	- \Big[ \, \frac{1}{2} \, a(u,u) - F(u) \, \Big]
	+ \data_\coarse(\TT_\coarse)^2.
\end{equation}
\begin{remark}
	Our definition follows \cite{dks16}, but is shifted to ensure $\energy(\TT_\coarse) \geq 0$ for all $\TT_\coarse \in \T$.
	This is important, since the GOAFEM analysis involves energy products.
\end{remark}
Recall that $u \in H^1_0(\Omega)$ solves the variational formulation~\eqref{eq:weakform} if and only if it minimizes the Dirichlet energy, i.e.,
\begin{align}
 \frac{1}{2} \, a(u,u) - F(u) = \inf_{v \in H^1_0(\Omega)} \Big[ \, \frac{1}{2} \, a(v,v) - F(v) \, \Big].
\end{align}
The same holds (with $H_0^1(\Omega)$ being replaced by $\SS_0^p(\Omega)$) for the Galerkin formulation~\eqref{eq:weakfemform}. By definition~\eqref{eq:def-oscillations} of the data resolutions terms, this proves $\energy(\TT_\coarse) \ge 0$. Moreover, from nestedness $\SS^p(\TT_\coarse) \subseteq \SS^p(\TT_\fine)$, we obtain the monotonicity $\energy(\TT_\fine) \le \energy(\TT_\coarse)$ for all $\TT_\coarse \in \T$ and $\TT_\fine \in \refine{\TT_\coarse}{}$.

\subsection{Verification of~\ref{enum:A:equiv}: Equivalence of energy and total error}\label{subsec:A5}

It is well-known from variational calculus that 
\begin{align}
 \energy(\TT_\coarse) = \frac{1}{2} \, \enorm{u - u_\coarse}^2 + \data_\coarse(\TT_\coarse)^2.
\end{align}
This and~\eqref{eq:data-osc-equivalence} prove~\ref{enum:A:equiv}. Moreover, for $\TT_\fine \in {\rm refine}(\TT_\coarse$), the Galerkin orthogonality proves the identity
\begin{align}\label{eq:galerkin:A5}
 \energy(\TT_\coarse) - \energy(\TT_\fine)
 = \frac{1}{2} \, \enorm{u_\coarse - u_\fine}^2 + \data_\coarse(\TT_\coarse)^2 - \data_\fine(\TT_\fine)^2,
\end{align}
which will be exploited below.

\subsection{Scott--Zhang projector}\label{subsec:scottzhang}

The key ingredient to prove~\ref{enum:A:lde}--\ref{enum:A:dle} is a slight variant~\cite[Lemma~3.5]{dks16} of the Scott--Zhang projector from~\cite{sz90}: Suppose $\TT_\coarse \in \T$ and $\TT_\fine \in \refine{\TT_\coarse}{}$. Let $\LL_\coarse$ denote the set of Lagrange nodes of of $\SS^p_0(\TT_\coarse)$. For each $z \in \LL_\coarse$, choose a simplex $\sigma_{\coarse,z} \in \TT_\coarse \cup \EE_\coarse$ subject to the following constraints:
	\begin{enumerate}[label=\rm(\alph*)]
		\item If $z \in T \in \TT_\coarse$ lies in the interior of $T$, choose $\sigma_{\coarse,z} = T$.
		\label{enum:transfer1}
		
		\item If $z \in E \in \EE_\coarse$ lies in the interior of $E$,  choose $\sigma_{\coarse,z} = E$.
		
		\item If $z \in \VV_\coarse$ with $z \in \area{\TT_\coarse \cap \TT_\fine}$ (resp.\ $z \in \area{\TT_\coarse \setminus \TT_\fine}$), choose $\sigma_{\coarse,z} = E \in \EE_\coarse$ with $E \subseteq \area{\TT_\coarse \cap \TT_\fine}$ (resp.\ $E \subseteq \area{\TT_\coarse \setminus \TT_\fine}$).
		\label{enum:transfer3}
	\end{enumerate}
For a Lagrange point $z \in \LL_\coarse$, let $\phi_{\coarse,z} \in \SS^p(\TT_\fine)$ be the corresponding nodal basis function, i.e., it holds that $\phi_{\coarse,z}(z') = \delta_{zz'}$ for all $z' \in \LL_\coarse$. Moreover, let $\psi_{\coarse,z} \in \PP^{p}(\sigma_{\coarse,z})$ be the corresponding dual basis function with respect to $L^2(\sigma_{\coarse,z})$, i.e., it holds that
\begin{equation}\label{eq:dualproperty}
	\int_{\sigma_{\coarse,z}} \psi_{\coarse,z} \phi_{\coarse,z'} \d{x} = \delta_{zz'}
	\quad \text{for all }
	z, z' \in \LL_\coarse.
\end{equation}
Then, we consider the Scott--Zhang projector $\transfer{\fine}{\coarse} : H^1(\Omega) \to \mathcal{S}^p(\TT_\coarse)$ defined by
\begin{equation}\label{eq:transferdef}
 \transfer{\fine}{\coarse} v
 := \sum_{z \in \LL_\coarse} \phi_{\coarse,z} \int_{\sigma_{\coarse,z}} \psi_{\coarse,z} v \d{x} 
 \quad \text{for all }  v \in H^1(\Omega).
\end{equation}
The following proposition collects the relevant properties of $\transfer{\fine}{\coarse}$.
We note that the definition guarantees that, for $v_\fine \in \SS^p(\TT_\fine)$, the restriction of $\transfer{\fine}{\coarse} v_\fine$ to $\Omega(\TT_\coarse \cap \TT_\fine)$ (resp.\ $\Omega(\TT_\coarse \setminus \TT_\fine)$) depends only on $v_\fine$ restricted to $\Omega(\TT_\coarse \cap \TT_\fine)$ (resp.\ $\Omega(\TT_\coarse \setminus \TT_\fine)$).
This is enforced by the choice \ref{enum:transfer3} of $\sigma_{\coarse,z}$.

\begin{proposition}\label{prop:szproperties}
Let $\TT_\coarse \in \T$ and $\TT_\fine \in \refine{\TT_\coarse}{}$. Let $\UU \in \{ \TT_\coarse \cap \TT_\fine \, , \, \TT_\coarse \setminus \TT_\fine \}$. Then, there hold the following assertions~\ref{enum:szproperties1}--\ref{enum:transferproperties3}, where $\Csz > 0$ depends only on $\Cmesh$ and $p$:
\begin{enumerate}[label=\rm(\roman*)]
\item \label{enum:szproperties1}
$\seminorm{\transfer{\fine}{\coarse} v}{H^1(T)} \leq \Csz \seminorm{v}{H^1(\TT_\coarse(T))}$
\quad for all $T \in \TT_\coarse$ and $v \in H^1(\Omega)$.
\item \label{enum:szproperties3}
$\norm{(1-\transfer{\fine}{\coarse}) v}{L^2(T)}{} \leq \Csz h_{T} \seminorm{v}{H^1(\TT_\coarse(T))}$
\quad for all $T \in \TT_\coarse$ and $v \in H^1(\Omega)$.
\item \label{enum:szproperties4}
$\norm{(1-\transfer{\fine}{\coarse}) v}{L^2(E)}{} \leq \Csz h_{E}^{1/2} \seminorm{v}{H^1(\TT_\coarse(E))}$
\quad for all $E \in \EE_\coarse$ and $v \in H^1(\Omega)$.
\item \label{enum:transferproperties1}
$(\transfer{\fine}{\coarse}v_\fine)|_{\area{\UU}}$ depends only on $v_\fine|_{\area{\UU}}$
\quad for all $v_\fine \in \mathcal{S}^p(\TT_\fine)$.
\item\label{enum:transferproperties2}
$\big(\transfer{\fine}{\coarse}v_\fine - v_\fine \big)|_T = 0$
\quad for all $T \in \TT_\coarse \cap \TT_\fine$ and all $v_\fine \in \mathcal{S}^p(\TT_\fine)$.
\item\label{enum:transferproperties3}
$\big(\transfer{\fine}{\coarse}v_\fine\big)|_\Gamma = 0$
\quad for all $v_\fine \in \mathcal{S}^p_0(\TT_\fine)$.
\end{enumerate}
\end{proposition}

\begin{proof}
The claims~\ref{enum:szproperties1}--\ref{enum:szproperties4} are proved in~\cite{sz90}. For~\ref{enum:transferproperties1}--\ref{enum:transferproperties3}, we refer to~\cite[Lemma~3.5]{dks16} (which directly transfers from $p = 1$ to $p \ge 1$).
\end{proof}

\subsection{Verification of~\ref{enum:A:lde}: Diamond estimate}\label{subsec:energy}
In order to prove the diamond estimate~\ref{enum:A:lde}, we employ the Scott--Zhang projector from the previous section.
The following lemma is proved in~\cite[Theorem~3.7]{dks16} (which directly transfers from $p=1$ to $p \geq 1$).

\begin{lemma}
Let $(\TT_\coarse,\TT_\fine; \TT_1, \ldots, \TT_m) \in \T^{m+2}$ be a diamond and $p\in\N$.
Then, the Scott--Zhang projectors $\transfer{\fine}{i}$ commute pairwise and the projection 
\begin{align}
 \transfer{\fine}{\circ} := \transfer{\fine}{1} \circ \ldots \circ \transfer{\fine}{m}: \mathcal{S}^p_0(\TT_\fine) \to \mathcal{S}^p_0(\TT_\coarse)
\end{align} 
is well-defined and satisfies that
\begin{equation}
 \seminorm{\transfer{\fine}{\circ} v_\fine}{H^1(\Omega)} \leq C \seminorm{v_\fine}{H^1(\Omega)}
 \quad \text{for all } v_\fine \SS^p_0(\TT_\fine),
\end{equation}
where $C > 0$ depends only on $\Cshapereg$ and $p$. Moreover, with $\Omega_i := \area{\TT_i \setminus \TT_\fine}$ for $i=1,\ldots,m$, it holds that
\begin{equation}\label{eq:transferareas}
 \transfer{\fine}{\circ} v_\fine =
		\left\{
		\begin{array}{ll}
			\transfer{\fine}{i} v_\fine & \textrm{ on } \Omega_i,\\
			v_\fine & \textrm{ on } \Omega \setminus \bigcup_{i=1}^m \Omega_i.
		\end{array}
		\right.
	\end{equation}
	\hfill $\square$
\end{lemma}

With this auxiliary result, we can prove the diamond estimate~\ref{enum:A:lde}.

\begin{proposition}\label{prop:lde}
The diamond estimate~{\rm\ref{enum:A:lde}} holds with a constant $\Clde > 0$ depending only on $\Cshapereg$, $p$, and $A$.
\end{proposition}

\begin{proof}[Sketch of proof]
The proof is split into three steps.

{\bf Step~1.} From the best approximation property of FEM solutions with respect to the energy norm and the stability of Scott--Zhang operators, we infer that
	\begin{equation*}
		\enorm[]{u_{\fine} - u_{i}}
		\simeq
		\enorm[]{u_{\fine} - \transfer{\fine}{i} u_{\fine}}
		\quad \text{for all }
		i=1,\ldots,m;
	\end{equation*}
	see also~\cite[Lemma~3.4]{dks16}.
	This equivalence holds also for $u_{\coarse}$ and $\transfer{\fine}{\circ}$ instead of $u_{i}$ and $\transfer{\fine}{i}$, respectively.
	Together with \eqref{eq:transferareas} and Proposition~\ref{prop:szproperties}\ref{enum:transferproperties2}, we obtain that
	\begin{equation}\label{eq:lde:step1}
		\enorm[]{u_{\fine} - u_{\coarse}}^2
		\simeq
		\enorm[]{u_{\fine} - \transfer{\fine}{\circ} u_{\fine}}^2
		\reff{eq:transferareas}{=}
		\sum_{i=1}^{m} \enorm[]{u_{\fine} - \transfer{\fine}{i} u_{\fine}}^2 \\
		\simeq
		\sum_{i=1}^{m} \enorm[]{u_{\fine} -  u_{i}}^2.
	\end{equation}

{\bf Step~2.} Let $\TT_\bullet \in \T$ and $\TT_\circ \in \refine{\TT_\bullet}{}$.
	Then, newest vertex bisection guarantees that
	\begin{equation*}
		|T'| \leq \tfrac{1}{2} |T|
		\quad \text{for all} ~
		T \in \TT_{\bullet} \setminus \TT_\circ
		\text{ and }
		T' \in \TT_{\circ} \setminus \TT_\bullet
		\text{ with }
		T' \subset T.
	\end{equation*}
	Together with the fact that $\jump{\f \cdot \normal}$ vanishes on all newly created edges, since $\div \f \in L^2(T)$ for every $T \in \TT_0$, this shows the equivalence
	\begin{equation}
	\label{eq:prop:lde}
		\big( 1 - \tfrac{1}{\sqrt{2}} \big) \data_\bullet(\TT_\bullet \setminus \TT_\circ)^2
		\leq
		\data_\bullet(\TT_\bullet)^2 - \data_\circ(\TT_\circ)^2
		\leq
		\data_\bullet(\TT_\bullet \setminus \TT_\circ)^2.
	\end{equation}

{\bf Step~3.} 	
We use \eqref{eq:prop:lde} on the meshes $\TT_\fine, \TT_i \in {\rm refine}(\TT_\coarse)$.
	This yields that
	\begin{align*}
		&\data_\coarse(\TT_\coarse)^2 - \data_\fine(\TT_\fine)^2
		\stackrel{\eqref{eq:prop:lde}}{\simeq}
		\data_\coarse(\TT_\coarse \setminus \TT_\fine,)^2
		=
		\data_\coarse\Big(\bigcup_{i=1}^m (\TT_i \setminus \TT_\fine) \Big)^2 
		\\& \quad
		=
		\sum_{i=1}^m \data_\coarse(\TT_i \setminus \TT_\fine)^2
		\stackrel{\eqref{eq:prop:lde}}{\simeq}
		\sum_{i=1}^m \big[ \, \data_i(\TT_i)^2 - \data_\fine(\TT_\fine)^2 \, \big].
	\end{align*}
Together with~\eqref{eq:lde:step1}, we see that
\begin{align*}
 \energy(\TT_\coarse) - \energy(\TT_\fine)
 &\reff{eq:galerkin:A5}{=} \frac{1}{2} \, \enorm{u_\coarse - u_\fine}^2 + \data_\coarse(\TT_\coarse)^2 - \data_\fine(\TT_\fine)^2
 \\&
 \reff{eq:lde:step1}\simeq \sum_{i=1}^m \big[ \, \frac{1}{2} \enorm[]{u_{i} -  u_{\fine}}^2 + \data_i(\TT_i) - \data_\fine(\TT_\fine)^2 \, \big]
 \reff{eq:galerkin:A5}{=} \sum_{i=1}^m \big[ \, \energy(\TT_i) - \energy(\TT_h) \, \big].
\end{align*}
This concludes the proof.
\end{proof}

\subsection{Verification of~\ref{enum:A:dle}: Local energy estimates for the estimator}\label{subsec:dle}
Since we have already verified \eqref{eq:galerkin:A5}, it suffices to show the discrete local estimates \ref{enum:A:dle} for the total error.

\subsubsection{\bfseries Discrete reliability}
We note that
\begin{equation*}
	\data_\coarse(\TT_\coarse)^2 - \data_\fine(\TT_\fine)^2
	\stackrel{\eqref{eq:prop:lde}}{\simeq}
	\data_\coarse(\TT_\coarse \setminus \TT_\fine)^2
	\stackrel{\eqref{eq:osc-eta}}{\leq}
	\eta_\coarse(\EE_\coarse \setminus \EE_\fine)^2.
\end{equation*}
Thus, the next proposition shows the upper bound in \ref{enum:A:dle}.
\begin{proposition}
	Let $\TT_\coarse \in \T$ and $\TT_\fine \in \refine{\TT_\coarse}{}$.
	Let $p \in \N$.
	Then, it holds that
	\begin{equation}
	\label{eq:discreteReliability}
		\enorm{u_\coarse - u_{\fine}}^2
		\leq
		\Cdrel \, \est_\coarse(\EE_\coarse \setminus \EE_\fine)^2.
	\end{equation}
	The constant $\Cdrel > 0$ depends only on $\Cshapereg$, $p$, and $A$.
\end{proposition}

\begin{proof}[Sketch of proof]
	Recall the Galerkin orthogonality
	\begin{equation}\label{eq:galerkinorth}
		\int_{\Omega} A \nabla (u_{\fine} - u_\coarse) \cdot \nabla v_\coarse \d{x}
		=
		0
		\qquad \text{for all }
		v_\coarse \in \mathcal{S}^p_0(\TT_\coarse).
	\end{equation}
	Therefore, we can insert $\transfer{\fine}{\coarse} (u_{\fine} - u_\coarse) \in \mathcal{S}^p_0(\TT_\coarse)$ into the bilinear form $a(\cdot,\cdot)$.
	With $v_\fine := (1-\transfer{\fine}{\coarse})(u_{\fine} - u_\coarse) \in \mathcal{S}^p_0(\TT_\fine)$, this yields that
	\begin{equation*}
		\enorm{u_{\fine} - u_\coarse}^2
		\refstack[]{\eqref{eq:galerkinorth}}{=}
		a(u_{\fine} - u_\coarse,(1 - \transfer{\fine}{\coarse}) (u_{\fine} - u_\coarse))\\
		=
		a(u_{\fine},v_\fine) - \int_{\Omega} A \nabla u_\coarse \cdot \nabla v_\fine \d{x}.
	\end{equation*}
	Using $\TT_\coarse$-elementwise integration by parts, we see that
	\begin{equation*}
	\label{eq:dlressential}
		\enorm{u_{\fine} - u_\coarse}^2
		=
		\!\!\! \sum_{T \in \TT_\coarse \setminus \TT_\fine} \int_T \big(f + \div \f + \div(A \nabla u_\coarse)\big) v_\fine \d{x}
		+ 
		\!\!\! \sum_{E \in \EE_\coarse^\Omega \setminus \EE_\fine^\Omega} \int_{E} \jump{ (A \nabla u_\coarse + \f) \cdot \normal} v_\fine \d{s}.
	\end{equation*}
	Standard estimates conclude \eqref{eq:discreteReliability}.
\end{proof}

\subsubsection{\bfseries Discrete efficiency}

The following proposition is proved along the lines of \cite[Proposition~2]{flop10} and adapts Verf\"urth's bubble function technique with cleverly chosen bubble functions.
We note that the idea goes back to the seminal works \cite{doerfler1996,mns2000}.
This result extends \cite[Lemma~4.3]{dks16} to polynomial degrees $p \geq 1$.

\begin{proposition}\label{prop:dle}
	Let $\TT_\coarse \in \T$ and $\TT_\fine \in \refine{\TT_\coarse}{}$.
	Let $p \in \N$.
	For $T \in \TT_\coarse \setminus \TT_\fine$ and $E \in \EE_\coarse \setminus \EE_\fine$ there hold the estimates
	\begin{align}
	\label{eq:prop:dle1}
		|E| \, \norm{\jump{(A \nabla u_\coarse + \f)\cdot\normal}}{L^2(E)}^2
		& \lesssim
		\norm{A^{1/2} \nabla (u_\coarse - u_{\fine})}{L^2(\TT_\coarse^{\rm red}(E))}^2
		+ \data_\coarse(\TT_\coarse^{\rm red}(E))^2 \\
	\nonumber
		&  \quad \qquad + \sum_{T \in \TT_\coarse^{\rm red}(E)} |T| \, \norm{f + \div (A \nabla u_\coarse + \f)}{L^2(T)}^2, \\
	\label{eq:prop:dle2}
		|T| \, \norm{f + \div (A \nabla u_\coarse + \f)}{L^2(T)}^2
		&\lesssim
		\norm{A^{1/2} \nabla (u_{\coarse} - u_{\fine})}{L^2(T)}^2 + \data_\coarse(T)^2.
		\hspace{-240pt}&
	\end{align}
	Together, there exists a constant $\Cdeff > 0$ such that there holds \emph{discrete local efficiency}
	\begin{equation}\label{eq:prop:dle9}
		\est^2_\coarse(\EE_\coarse \setminus \EE_\fine)
		\leq
		\Cdeff \big[ \, \enorm{u_\coarse - u_{\fine}}^2 + \data_\coarse(\TT_\coarse \setminus \TT_\fine)^2 \, \big].
	\end{equation}
	The constant $\Cdeff$ depends only on $\Cshapereg$, $A$, and $p$.
\end{proposition}

\begin{proof}[Sketch of proof]
	We use ideas from \cite[Lemma~11]{hh2} and employ Verf\"urth's bubble function technique with discrete, conforming bubbles.
	To this end, let $\TT_\coarse \in \T$ and $\TT_\fine \in {\rm refine}(\TT_\coarse)$.
	For $z \in \VV_\coarse$, let $\phi_{\coarse,z} \in \SS^1(\TT_\coarse)$ be the piece-wise affine hat function.
	An element $T \in \TT_\coarse \setminus \TT_\fine$ has at least one edge $E \subset \partial T$ with $E \in \EE_\coarse \setminus \EE_\fine$.
	We denote the midpoint of this edge by $z' \in \VV_\fine$ and the vertex opposite to $E$ by $z \in \VV_\coarse$.
	We then define the corresponding edge and element bubble functions as
	\begin{equation}\label{eq:bubblefunc}
		b_{E} := \phi_{\fine, z'} \in \mathcal{S}^1(\TT_\fine)
		~\text{ and }~
		b_{T} := \phi_{\coarse,z} \phi_{\fine, z'} \in \mathcal{S}^2_0(\TT_\fine),
	\end{equation}
	respectively.
	The estimates \eqref{eq:prop:dle1} for $p \geq 1$ and \eqref{eq:prop:dle2} for $p \geq 2$ follow directly from the bubble function technique with $b_E$ and $b_T$, respectively.
	For \eqref{eq:prop:dle2} with $p=1$, it holds that
	\begin{equation*}
		|T| \, \norm{f + \div (A \nabla u_\coarse + \f)}{L^2(T)}^2
		=
		|T| \, \norm{f + \div \f}{L^2(T)}^2
		\leq
		\data_\coarse(T)^2.
	\end{equation*}
	Combining~\eqref{eq:prop:dle1}--\eqref{eq:prop:dle2}, we conclude~\eqref{eq:prop:dle9}.
\end{proof}

\subsection{Proof of Theorem~\ref{theorem:instanceoptimality}}\label{sec:proofafem}
In the last sections, we have verified \ref{enum:A:mark}--\ref{enum:A:equiv} for the primal problem~\eqref{eq:intro:primal}.
From Proposition~\ref{theorem:abstractio}, we thus infer instance optimality (Theorem~\ref{theorem:instanceoptimality}) of Algorithm~\ref{alg:abstractafem}.
Clearly, the same results hold for the dual problem~\eqref{eq:intro:dual}, which differs from~\eqref{eq:intro:primal} only through the right-hand side $G$ instead of $F$.

%% file: 05_goafem_proof.tex
\section{Proof of Theorem~\ref{theorem:goafem}}
\label{sec:proofgoafem}

The key observation for the proof of Theorem~\ref{theorem:goafem} is that the proposed GOAFEM (Algorithm~\ref{alg:abstractgoafem}) is simultaneously instance optimal for both, the primal and the dual error estimate. Since the properties~\ref{enum:A:mon}--\ref{enum:A:equiv} are already verified for primal and dual problem (see Section~\ref{sec:auxiliary}), it only remains to show that the marking strategy of GOAFEM (Algorithm~\ref{alg:goafemmarking}) satisfies~\ref{enum:A:mark} for both, the primal and the dual error estimator.

\begin{lemma}\label{lemma:goafemmarking}
For $\TT_\coarse \in \T$, let $\MM_\coarse \subseteq \EE_\coarse$ be set of marked edges generated by  Algorithm~\ref{alg:goafemmarking}. Then, it holds that $\MM_\coarse \neq 0$ as well as
\begin{align}
 \label{eq:goafemmmarking:primal}
 \est_\coarse (\tail{\MM_\coarse}{\coarse})
 &\ge C \, \#\MM_\coarse \max_{E \in \EE_\coarse} \est_\coarse(\tail{E}{\coarse}), 
 \\
 \label{eq:goafemmmarking:dual}
 \est_\coarse^\ast (\tail{\MM_\coarse}{\coarse})
 &\ge C \, \#\MM_\coarse \max_{E \in \EE_\coarse} \est_\coarse^\ast(\tail{E}{\coarse}),
\end{align}
where $C > 0$ depends only on $\Cmin > 0$.
\end{lemma}

\begin{proof}
According to~\cite[Proposition 5.1]{dks16}, Algorithm~\ref{alg:marking} guarantees that
\begin{align*}
 \overline{\MM}_\coarse &\neq \emptyset
 & \text{with } \qquad
 \est_\coarse (\tail{\overline{\MM}_\coarse}{\coarse})
 &\ge \Cmark \, \#\overline{\MM}_\coarse \max_{E \in \EE_\coarse} \est_\coarse(\tail{E}{\coarse}),
 \\
 \overline{\MM}_\coarse^\ast &\neq \emptyset
 & \text{with } \qquad
 \est_\coarse^\ast (\tail{\overline{\MM}_\coarse^\ast}{\coarse})
 &\ge \Cmark \, \#\overline{\MM}_\coarse^\ast \max_{E \in \EE_\coarse} \est_\coarse^\ast(\tail{E}{\coarse}).
\end{align*}
Without loss of generality, we may assume that $\MM_{\rm min} = \overline{\MM}_\coarse$ and $\MM_{\rm max} = \overline{\MM}_\coarse^\ast$. Recall that $\#\MM_{\rm min} \le \#\MM_{\rm max}$ and $\MM_\coarse := \MM_{\rm min} \cup \MM_{\rm max}' \neq \emptyset$, where 
\begin{align*}
 \MM_{\rm max}' \subseteq \MM_{\rm max}
 \quad \text{with} \quad
 \#\MM_{\rm max}' = n := \min \big\{ \#\MM_{\rm max}, \max \{ 1, \lfloor \Cmin \#\MM_{\rm min} \rfloor \} \big\}.
\end{align*}
First, note that $\MM_{\rm max} \supseteq \MM_{\rm max}' \neq \emptyset$.
With Proposition~\ref{prop:marking}, it follows that
\begin{align*}
 \est_\coarse (\tail{\MM_\coarse}{\coarse})
 \ge \est_\coarse (\tail{\MM_{\rm min}}{\coarse})
 &\ge \Cmark \, \#\MM_{\rm min} \max_{E \in \EE_\coarse} \est_\coarse(\tail{E}{\coarse}),
 \\
 \est_\coarse^\ast (\tail{\MM_\coarse}{\coarse})
 \ge \est_\coarse^\ast (\tail{\MM_{\rm max}'}{\coarse})
 &\ge \Cmark \, \#\MM_{\rm max}' \max_{E \in \EE_\coarse} \est_\coarse^\ast(\tail{E}{\coarse}),
\end{align*}
In view of~\eqref{eq:goafemmmarking:primal}--\eqref{eq:goafemmmarking:dual}, it only remains to prove that $\#\MM_\coarse \lesssim \#\MM_{\rm min} \lesssim \#\MM_{\rm max}'$, where the hidden constants depend only on $\Cmin$. First, note that
\begin{align*}
 \#\MM_\coarse
 &\le \#\MM_\textrm{min} + \#\MM_\textrm{max}' \\
 &\le \#\MM_\textrm{min} + \max \{ 1, \lfloor \Cmin \#\MM_{\rm min} \rfloor \}
 \le \max \{ 2, (1 + \Cmin) \} \, \#\MM_{\rm min}.
\end{align*} 
This already guarantees that
\begin{align*}
 \frac{1}{\max \{ 2, (1 + \Cmin) \}} \, \#\MM_\coarse \le \#\MM_{\rm min} 
 \le \#\MM_\coarse.
\end{align*}
To estimate $\#\MM_{\rm max}'$, we consider two cases:

{\bf Case 1.} If $n = \#\MM_{\rm max} \le \max \{ 1, \lfloor \Cmin \#\MM_{\rm min} \rfloor \}$, then $\MM_{\rm max} = \MM_{\rm max}'$. Therefore,
\begin{align*}
 \frac{1}{\max \{ 2, (1 + \Cmin) \}} \, \#\MM_\coarse \le \#\MM_{\rm min} 
 \le \#\MM_{\rm max} \le \#\MM_\coarse.
\end{align*}

{\bf Case 2.} If $n = \max \{ 1, \lfloor \Cmin \#\MM_{\rm min} \rfloor \} < \#\MM_{\rm max}$, then $\#\MM_{\rm max}^\prime \geq 1$ leads to
\begin{align*}
		\#\MM_{\rm max}^\prime
		\ge
		\lfloor \Cmin \#\MM_{\rm min} \rfloor
		\geq
		\Cmin \#\MM_{\rm min} - 1
		\geq
		\Cmin \#\MM_{\rm min} - \#\MM_\textrm{max}^\prime.
\end{align*}
Hence, we see that
\begin{align*}
 \frac{1}{\max \{ 2, (1 + \Cmin) \}} \, \#\MM_\coarse 
 \le \#\MM_{\rm min} 
 \le \frac{2}{\Cmin} \, \#\MM_{\rm max}'
 \le \frac{2}{\Cmin} \, \MM_\coarse. 
\end{align*}
In any case, we see that $\#\MM_\coarse \simeq \#\MM_{\rm min} \simeq \#\MM_{\max}'$. This concludes the proof.
\end{proof}

\begin{proof}[\bfseries Proof of Theorem~\ref{theorem:goafem}]
Section~\ref{sec:auxiliary} shows that \ref{enum:A:mon}--\ref{enum:A:equiv} are satisfied for the primal and the dual problem.
Lemma~\ref{lemma:goafemmarking} shows that $\MM_\coarse$ from Algorithm~\ref{alg:goafemmarking} satisfies \ref{enum:A:mark} simultaneously for both estimators.
Proposition~\ref{theorem:abstractio} thus implies instance optimality of Algorithm~\ref{alg:abstractgoafem} for the primal and dual energy. In explicit terms, there exists $C, C^\ast > 0$ such that for all $\ell \in \N_0$
\begin{align*}
 \forall \, \TT_\coarse \in {\rm refine}(\TT_0) : \quad
 &\big[ \, C \, \#(\TT_\coarse \backslash \TT_0) \le \#(\TT_\ell \backslash \TT_0) \, \Longrightarrow \,
 \energy(\TT_\ell)
 \le \energy(\TT_\coarse) \, \big],
 \\
 \forall \, \TT_{\coarse^\ast} \in {\rm refine}(\TT_0) : \quad
 &\big[ \, C^\ast \, \#(\TT_{\coarse^\ast} \backslash \TT_0) \le \#(\TT_\ell \backslash \TT_0) \, \Longrightarrow \,
 \energy^\ast(\TT_\ell)
 \le \energy^\ast(\TT_{\coarse^\ast}) \, \big],
\end{align*}
where $\energy(\cdot)$ denotes the energy~\eqref{eq:def:energy} for the primal problem, and $\energy^\ast(\cdot)$ denotes the energy for the dual problem. Obviously, this leads to
\begin{align*}
 \forall \, \TT_\coarse, \TT_{\coarse^\ast} \in {\rm refine}(\TT_0) : \quad
 \big[ \, \max\{C, C^\ast\} \, \max\{ &\#(\TT_\coarse \backslash \TT_0) \,,\, \#(\TT_{\coarse^\ast} \backslash \TT_0)\} \le \#(\TT_\ell \backslash \TT_0) \, 
 \\&\quad
 \Longrightarrow \,
 \energy(\TT_\ell) \, \energy^\ast(\TT_\ell)
 \le \energy(\TT_\coarse) \, \energy^\ast(\TT_{\coarse^\ast}) \, \big],
\end{align*}
Using the equivalence~\ref{enum:A:equiv} of energy and total error (for primal and dual problem), we conclude the proof.
\end{proof}

%% file: 06_numerics.tex
\section{Numerical experiments}
\label{sec:numerics}

We conclude this work with some numerical experiments performed in MATLAB, where our implementation builds on the codes provided in~\cite{p1afem} for $p=1$ and~\cite{p2afem} for $p=2$.
For the modified maximum criterion (Algorithm~\ref{alg:marking}), we have implemented a recursive variant proposed in~\cite[Algorithm~7.2]{dks16}.

\subsection{Adaptive FEM with Z-shaped domain}
\label{subsec:afem-numerics}

We consider the problem
\begin{subequations}
	\label{eq:numerics:afem-problem}
	\begin{align}
		\hspace*{20mm} -\Delta u &= 1 && \textrm{ in } \Omega :=(-1,1)^2 \setminus {\rm conv}\{(0,0), (-1,0), (-1,-1) \}, \hspace*{10mm} \\
		u &= 0  && \textrm{ on } \Gamma \coloneqq \partial \Omega, \hspace*{10mm}
	\end{align}
\end{subequations}
where ${\rm conv}(\cdot)$ denotes the convex hull and $\Omega$ is the Z-shaped domain from Figure~\ref{fig:mesh-compare}.
This problem is solved with the instance-optimal algorithm from~\cite{dks16}, i.e., Algorithm~\ref{alg:abstractafem}. Moreover, we compare the results with a rate-optimal algorithm, which builds on an edge-based D\"orfler marking criterion \cite{doerfler1996}:
Find a subset $\MM_\coarse \subseteq \EE_\coarse$ with minimal cardinality such that
\begin{equation}
\label{eq:doerfler}
	\theta \, \mu_\coarse(\EE_\coarse)^2 \leq \mu_\coarse(\MM_\coarse)^2
\end{equation}
for an edge-based error estimator $\mu_\coarse: \EE_\coarse \to \R$.
Note that uniform refinement corresponds to $\theta = 1$ in~\eqref{eq:doerfler} but $\vartheta = 0$ in Algorithm~\ref{alg:marking}.
Therefore, we set $\theta = 1 - \vartheta$ in the following to account for the different interpretations of the marking parameters. We note that both adaptive strategies only differ by the marking criterion. Throughout, we consider $\vartheta = 0.5$ and $\Cmin = 1$.

In Figure~\ref{fig:error-estimator-zshape},  we visualize the edge-based residual error estimator $\eta_\ell$ and the energy error $\enorm{u-u_\ell}$.
Experimentally, both strategies lead to optimal convergence of error and error estimator with rate $(\#\TT_\ell)^{-p/2}$.
We stress that this is mathematically guaranteed for the maximum strategy, while available results for the D\"orfler marking criterion require that $\theta$ is sufficiently small.
Since the exact solution $u$ is unknown, we extrapolate $\enorm{u}$ from the computed values $\enorm{u_\ell}$ and use the Galerkin orthogonality~\eqref{eq:galerkinorth} to obtain
\begin{equation*}
	\enorm{u-u_\ell}^2
	=
	\enorm{u}^2 - \enorm{u_\ell}^2.
\end{equation*}%
Since~\cite{dks16} does not provide any numerical experiments, we also give qualitative plots of the resulting meshes in Figure~\ref{fig:mesh-compare}.
Both strategies mark edges near the re-entrant corner.
For the D\"orfler marking criterion~\eqref{eq:doerfler}, edges in the interior are refined mostly by mesh closure.
Instead, the modified maximum criterion (Algorithm~\ref{alg:marking}) also marks edges in the interior that have long tails.

\begin{figure}
	\centering
	\hfill
	\begin{subfigure}[t]{0.45\textwidth}
		\centering
		\includegraphics[width=\linewidth]{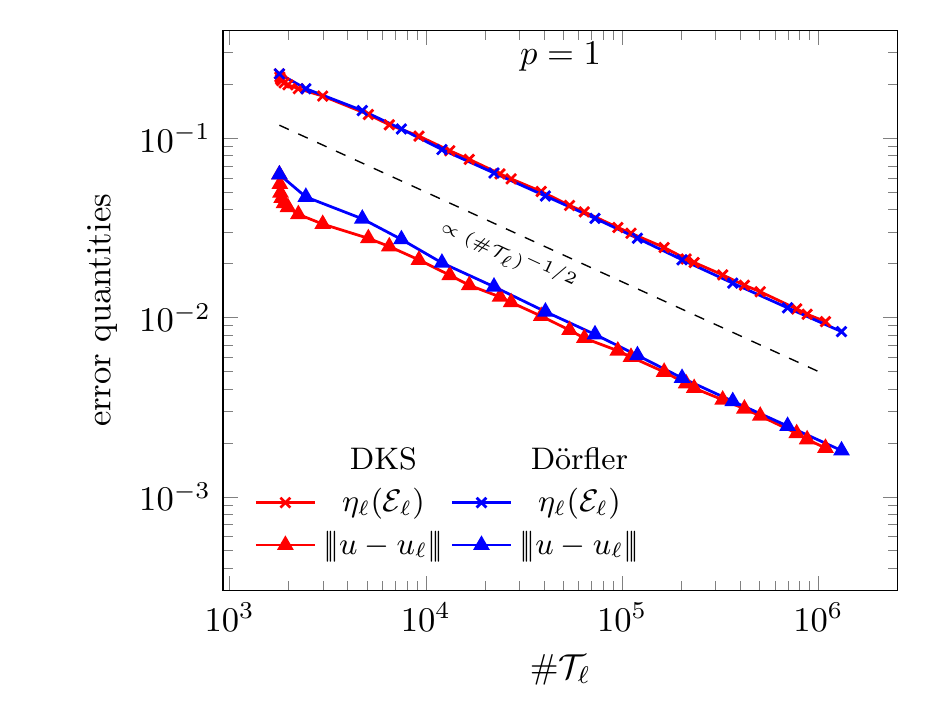}
	\end{subfigure}
	\hfill
	\begin{subfigure}[t]{0.45\textwidth}
		\centering
		\includegraphics[width=\linewidth]{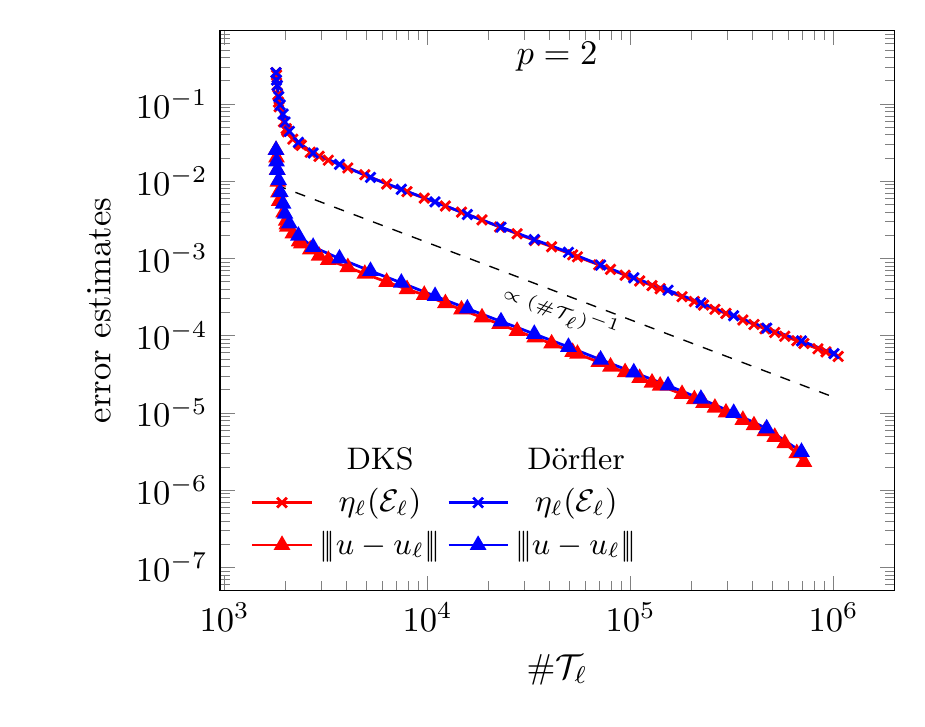}
	\end{subfigure}
	\hfill\hphantom{*}
	\caption{Error estimators and energy error for problem~\eqref{eq:numerics:afem-problem} with FEM of order $p=1$ (left) and $p=2$ (right). We compare the modified maximum criterion (DKS) with the D\"orfler marking~\eqref{eq:doerfler} for $\vartheta = 0.5$.}
	\label{fig:error-estimator-zshape}
\end{figure}

\begin{figure}
	\centering
	\hfill
	\begin{subfigure}[b]{0.49\textwidth}
		\centering
		\includegraphics[width=0.7\linewidth]{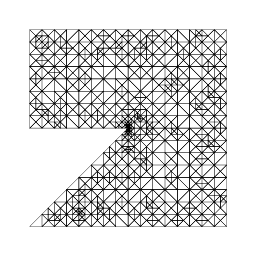}
		\caption{Modified maximum marking, $\#\TT_{10} = 1204$.}
	\end{subfigure}
	\hfill
	\begin{subfigure}[b]{0.49\textwidth}
		\centering
		\includegraphics[width=0.7\linewidth]{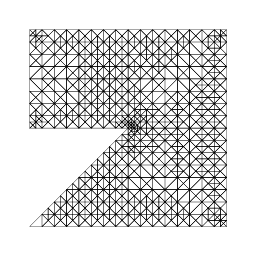}
		\caption{D\"orfler marking~\eqref{eq:doerfler}, $\#\TT_{7} = 1196$.}
	\end{subfigure}
	\hfill\hphantom{*}
	\caption{Comparison of meshes generated by 	
	 Algorithm~\ref{alg:abstractafem} with $p=1$, $\vartheta = 0.5$, and different marking strategies for problem~\eqref{eq:numerics:afem-problem}.}
	\label{fig:mesh-compare}
\end{figure}

\subsection{Goal-oriented AFEM}
\label{subsec:goafem-numerics}
The following numerical example empirically shows how the proposed goal-oriented adaptivity (Algorithm~\ref{alg:abstractgoafem}) handles possible singularities.
We consider a problem proposed in \cite{ms2008}, where the primal problem reads
\begin{subequations}
\label{eq:numerics:problem}
\begin{align}
	\hspace*{20mm} -\Delta u &= \div \f &&\hspace*{-25mm} \textrm{ in } \Omega :=(0,1)^2, \hspace*{10mm} \\
	u &= 0  &&\hspace*{-25mm} \textrm{ on } \Gamma \coloneqq \partial (0,1)^2, \hspace*{10mm}
\end{align}
\end{subequations}
with
\begin{equation*}
	\f(x) = \begin{cases}
	(1,0)^\top & \text{if } x \in T_F := \set{x \in \Omega}{x_1 + x_2 \leq 1/2}, \\
	(0,0)^\top & \text{else}.
	\end{cases}
\end{equation*}
With $T_G := \set{x \in \Omega}{x_1 + x_2 \geq 3/2}$ and
\begin{equation*}
	g = 0
	\qquad \text{and} \qquad
	\g(x) = \begin{cases}
	(1,0)^\top & \text{if } x \in T_G, \\
	(0,0)^\top & \text{else},
	\end{cases}
\end{equation*}%
the goal functional from~\eqref{eq:intro:goal} takes the form
\begin{equation*}
	G(v) = \int_{T_G} \frac{\partial v}{\partial x_1} \d{x}
	\qquad
	\text{for } v \in H_0^1(\Omega).
\end{equation*}
The initial triangulation $\TT_0$ with the subsets $T_F$ and $T_G$, together with approximations to the primal and dual solution can be seen in Figure~\ref{fig:goafem-solution}.
In particular, it is visible that the singularity of $u$ and $u^\ast$ are well separated so that optimal convergence rates can only be achieved if both singularities are appropriately resolved.

\begin{figure}
	\centering
	\hfill
	\begin{subfigure}[b]{0.24\textwidth}
		\centering
		\includegraphics[width=\linewidth]{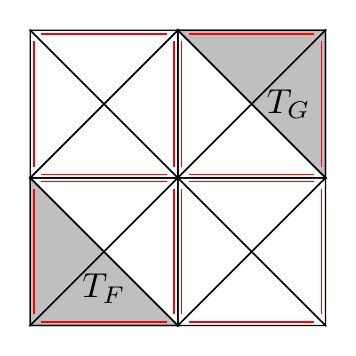}
	\end{subfigure}
	\vspace{-20pt}
	\begin{subfigure}[b]{0.37\textwidth}
		\centering
		\includegraphics[width=\linewidth]{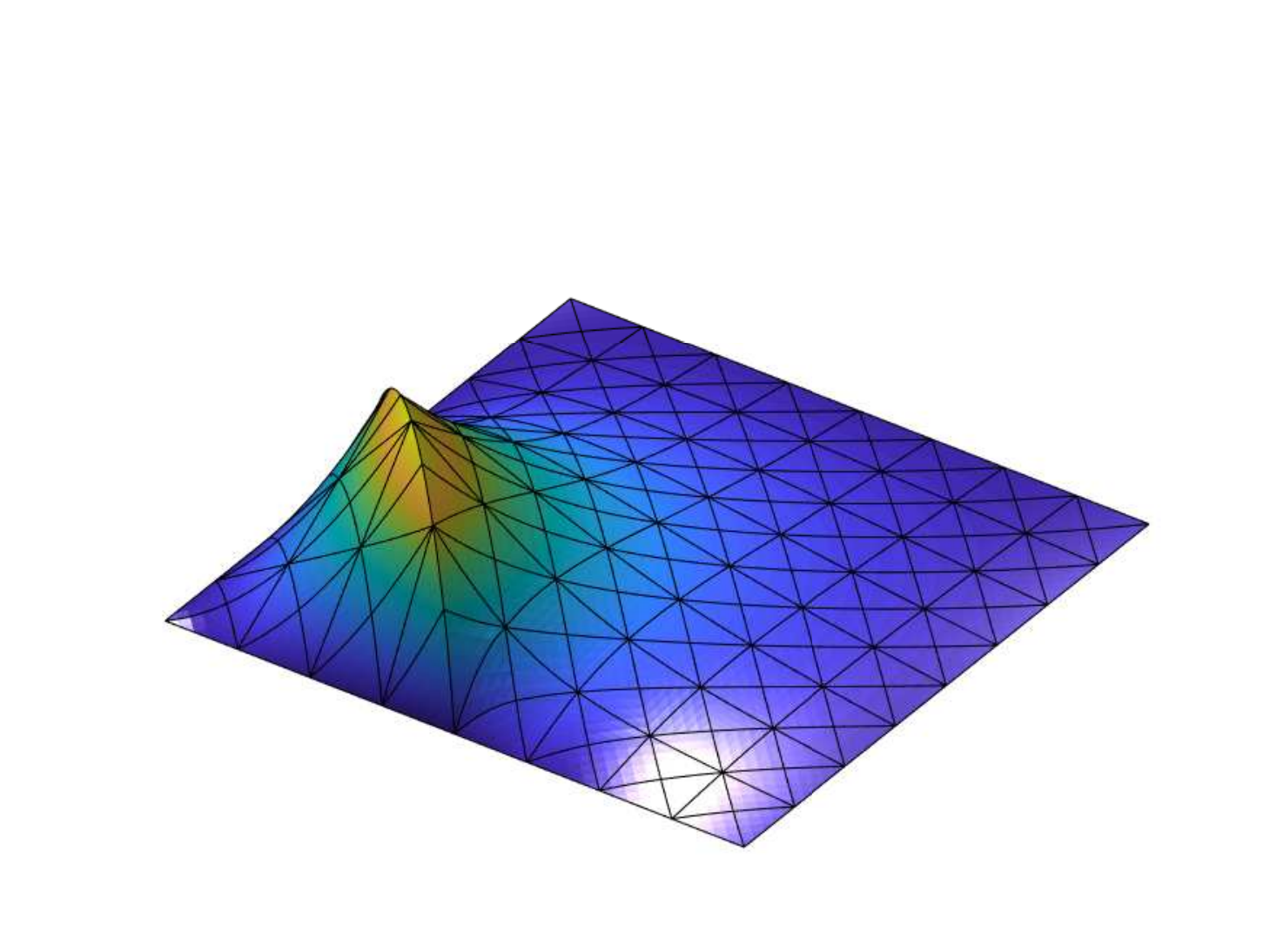}
	\end{subfigure}
	\begin{subfigure}[b]{0.37\textwidth}
		\centering
		\includegraphics[width=\linewidth]{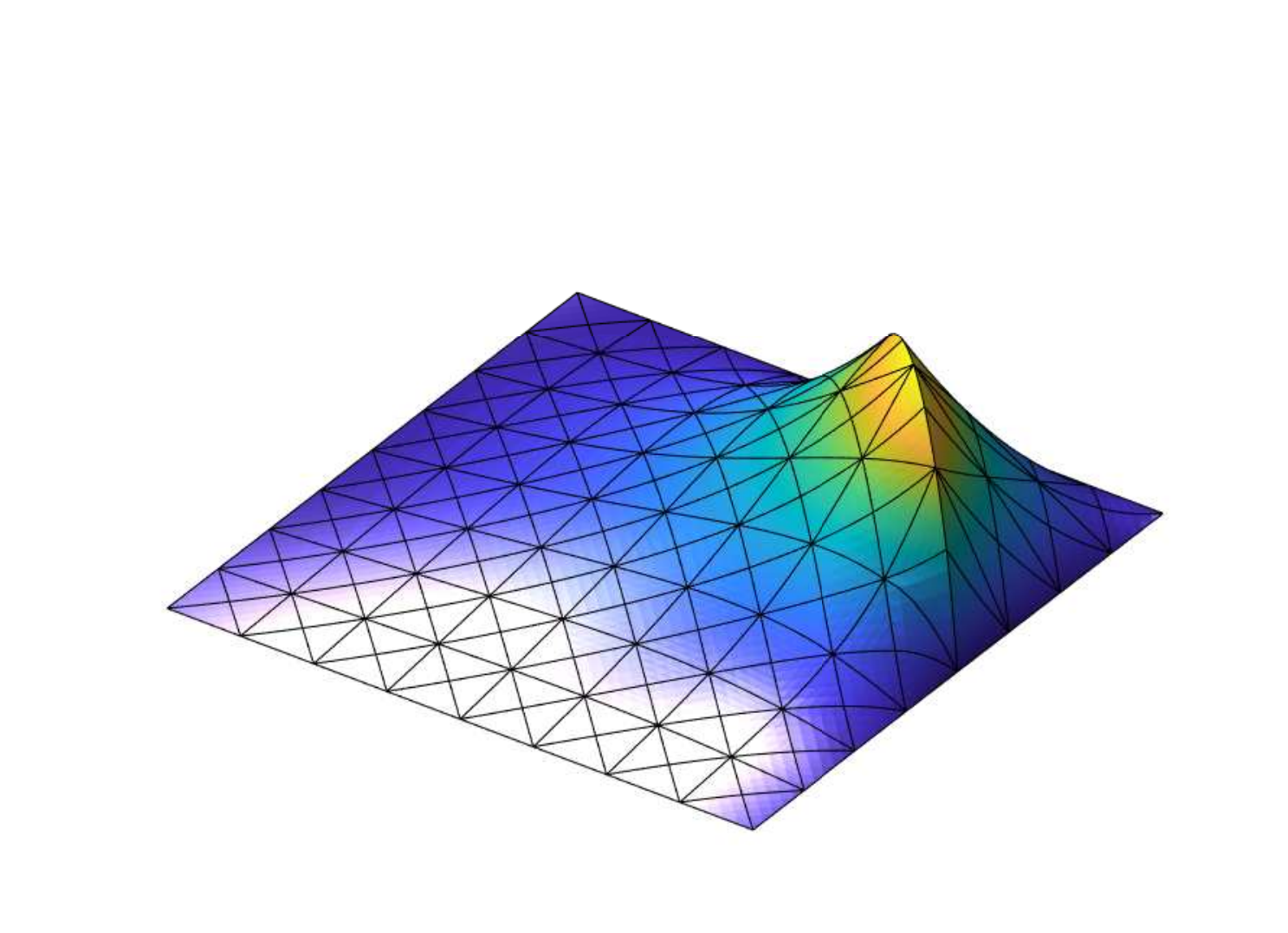}
	\end{subfigure}
	\hfill\hphantom{*}
	\caption{Initial mesh (left), and qualitative plot of primal solution (middle) and negative dual solution (right) of \eqref{eq:numerics:problem}.
	In the initial mesh, the sets $T_F$ and $T_G$ are highlighted in grey and the reference edge of each triangle is highlighted by a red line.}
	\label{fig:goafem-solution}
\end{figure}

In addition to Algorithm~\ref{alg:abstractgoafem}, we investigate the rate-optimal algorithms from \cite{ms2008,goafem,bet2011}.
These build on the D\"orfler marking criterion~\eqref{eq:doerfler}.
For the convenience of the reader, we briefly outline these marking strategies:
\begin{itemize}
	\item In \cite{ms2008}, the D\"orfler criterion~\eqref{eq:doerfler} is employed separately for $\mu_\coarse = \eta_\coarse$ as well as $\mu_\coarse = \eta_\coarse^\ast$ and thus provides sets $\overline{\MM}_\coarse, \overline{\MM}_\coarse^\ast \subseteq \TT_\coarse$.
	Then $\MM_\coarse$ is chosen as the smaller set $\MM_\coarse := \arg\min \{ \overline{\MM}_\coarse, \overline{\MM}_\coarse^\ast \}$.
	
	\item In \cite{goafem}, one proceeds analogously, but chooses $\MM_\coarse := \underline{\MM}_\coarse \cup \underline{\MM}_\coarse^\ast$, where $\underline{\MM}_\coarse \subseteq \overline{\MM}_\coarse$ and $\underline{\MM}_\coarse^\ast \subseteq \overline{\MM}_\coarse^\ast$ satisfy $\#\underline{\MM}_\coarse = \#\underline{\MM}_\coarse^\ast = \min \{ \#\overline{\MM}_\coarse, \#\overline{\MM}_\coarse^\ast \}$.
	
	\item For \cite{bet2011}, the D\"orfler criterion~\eqref{eq:doerfler} is employed for
	\begin{equation*}
		\mu_\coarse(E)^2 := \eta_\coarse(E)^2 \eta^\ast_\coarse(\EE_\coarse)^2 + \eta_\coarse(\EE_\coarse)^2 \eta^\ast_\coarse(E)^2.
	\end{equation*} 
\end{itemize}

In Figure~\ref{fig:error-estimator}, we compare the products of error estimates for the primal and dual problem for the above marking strategies.
For $p=1,2$ and all four strategies, this product decays with optimal rate $(\#\TT_\ell)^{-p}$.
Furthermore, Figure~\ref{fig:mesh-size} gives qualitative plots of the local mesh-size.

\begin{remark}
	Figures~\ref{fig:error-estimator-zshape} and~\ref{fig:error-estimator} underline that our instance optimal algorithms achieve the same (optimal) rate as the rate optimal algorithms, as can be expected from the discussion in the introduction.
	Note that it is difficult to show instance optimality directly in numerical examples, since the computational cost of computing the optimal mesh in~\eqref{eq:intro:instance} grows exponentially in the number of refined edges.
\end{remark}

\begin{figure}
	\centering
	\hfill
	\begin{subfigure}[t]{0.45\textwidth}
		\centering
		\includegraphics[width=\linewidth]{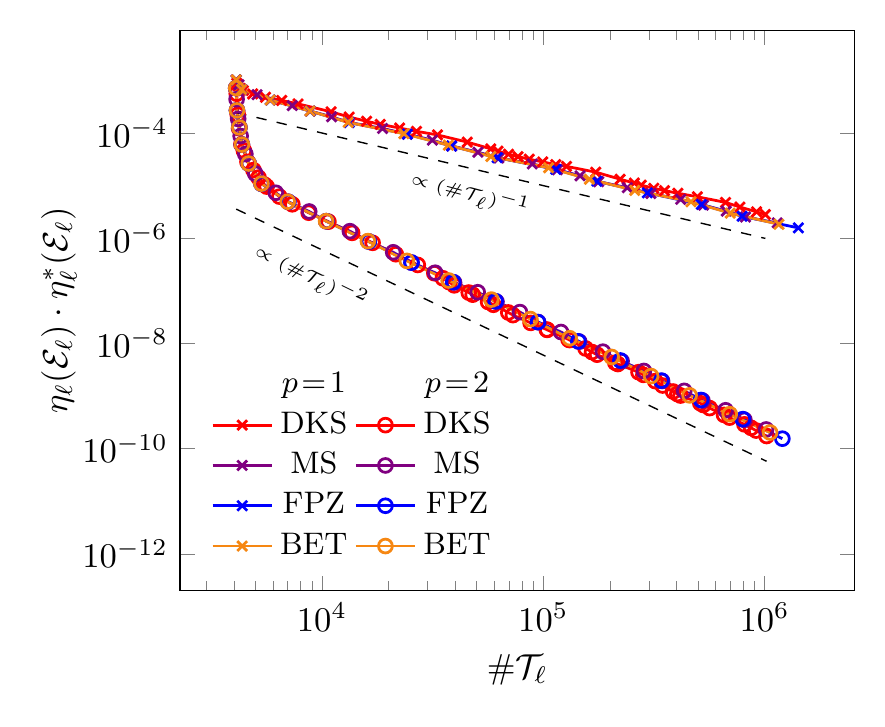}
		\caption{Different marking strategies in Algorithm~\ref{alg:abstractgoafem}, where we compare Algorithm~\ref{alg:goafemmarking} (DKS) and \cite{ms2008,goafem,bet2011} with marking parameter $\vartheta = \theta = 0.5$.}
	\end{subfigure}
	\hfill
	\begin{subfigure}[t]{0.45\textwidth}
		\centering
		\includegraphics[width=\linewidth]{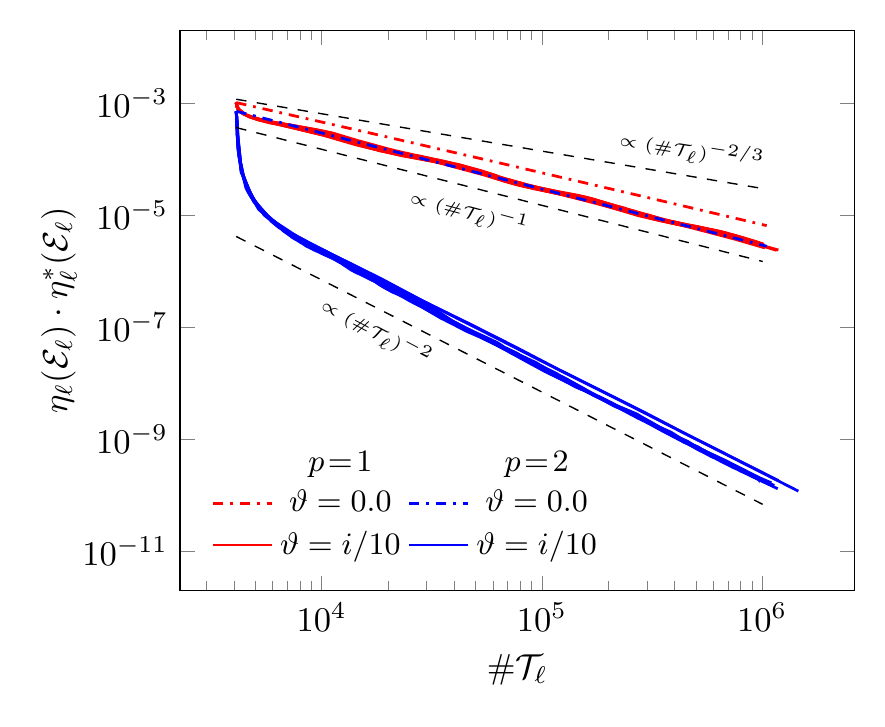}
		\caption{Comparison of estimates for Algorithm~\ref{alg:goafemmarking} (DKS) and different parameters $\vartheta$.
		For the adaptive case $\vartheta \in (0,1)$, we chose $\vartheta = i/10$ with $i=1,\ldots,9$.}
	\end{subfigure}
	\hfill\hphantom{*}
	\caption{Product of error estimates for the solution of problem~\eqref{eq:numerics:problem} with Algorithm~\ref{alg:abstractgoafem} and FEM of order $p=1$ and $p=2$.}
	\label{fig:error-estimator}
\end{figure}

\begin{figure}
	\centering
	\hfill
	\begin{subfigure}[c]{0.22\textwidth}
		\centering
		\includegraphics[trim=70 10 70 10, clip=true, width=\linewidth]{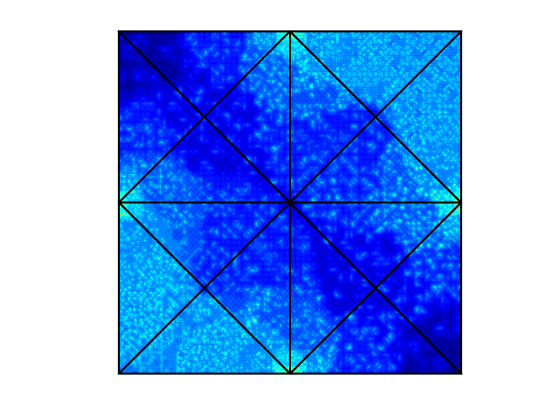}
		\caption{$\#\TT_{22} = 53997$.}
	\end{subfigure}
	\hfill
	\begin{subfigure}[c]{0.22\textwidth}
		\centering
		\includegraphics[trim=70 10 70 10, clip=true, width=\linewidth]{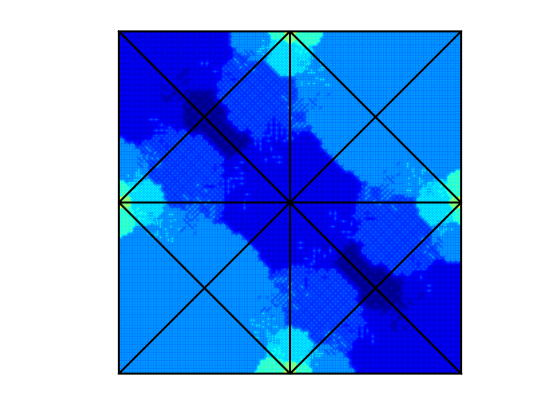}
		\caption{$\#\TT_{18} = 52596$.}
	\end{subfigure}
	\hfill
	\begin{subfigure}[c]{0.22\textwidth}
		\centering
		\includegraphics[trim=70 10 70 10, clip=true, width=\linewidth]{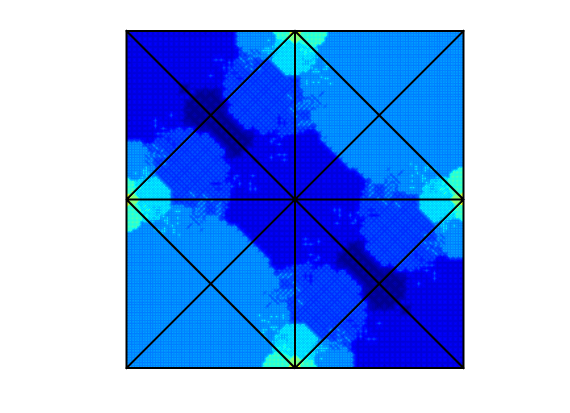}
		\caption{$\#\TT_{9} = 52078$.}
	\end{subfigure}
	\hfill
	\begin{subfigure}[c]{0.22\textwidth}
		\centering
		\includegraphics[trim=70 10 70 10, clip=true, width=\linewidth]{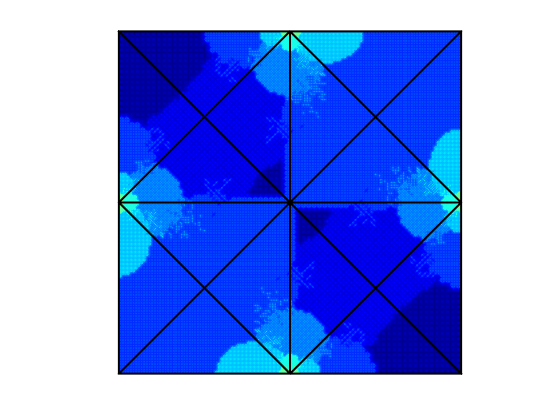}
		\caption{$\#\TT_{10} = 80142$.}
	\end{subfigure}
	\hfill
	\begin{subfigure}[c]{0.06\textwidth}
		\centering
		\includegraphics[trim=350 10 0 10, clip=true, width=\linewidth]{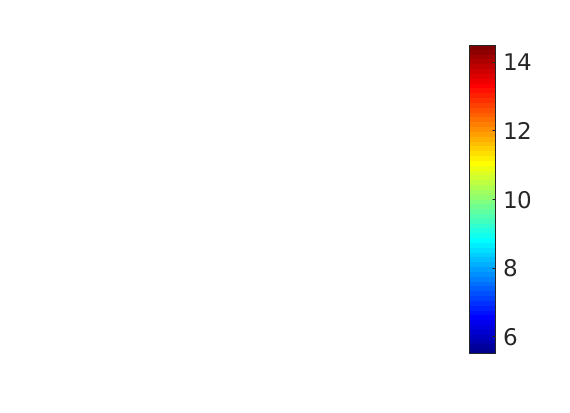}
	\end{subfigure}
	\hfill\hphantom{*}
	\\
	\hfill
	\begin{subfigure}[c]{0.22\textwidth}
		\centering
		\includegraphics[trim=70 10 70 10, clip=true, width=\linewidth]{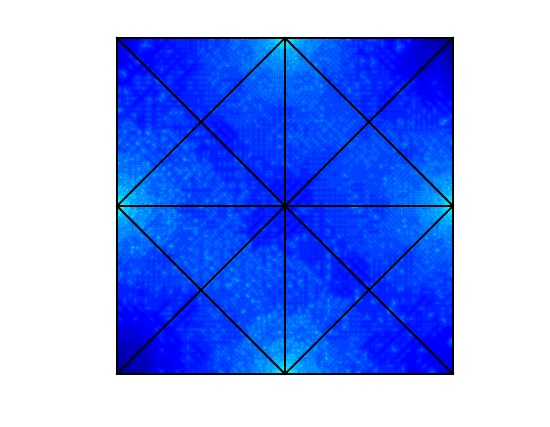}
		\caption{$\#\TT_{33} = 51619$.}
	\end{subfigure}
	\hfill
	\begin{subfigure}[c]{0.22\textwidth}
		\centering
		\includegraphics[trim=70 10 70 10, clip=true, width=\linewidth]{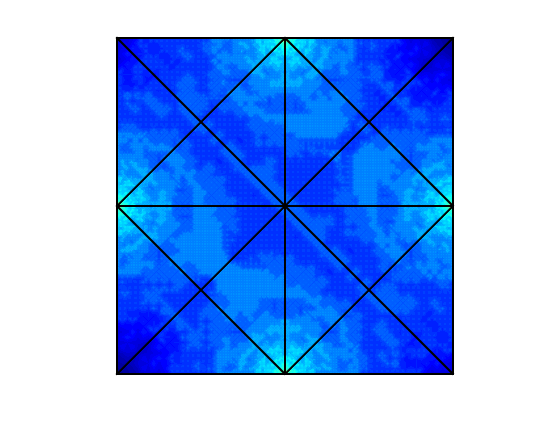}
		\caption{$\#\TT_{24} = 55118$.}
	\end{subfigure}
	\hfill
	\begin{subfigure}[c]{0.22\textwidth}
		\centering
		\includegraphics[trim=70 10 70 10, clip=true, width=\linewidth]{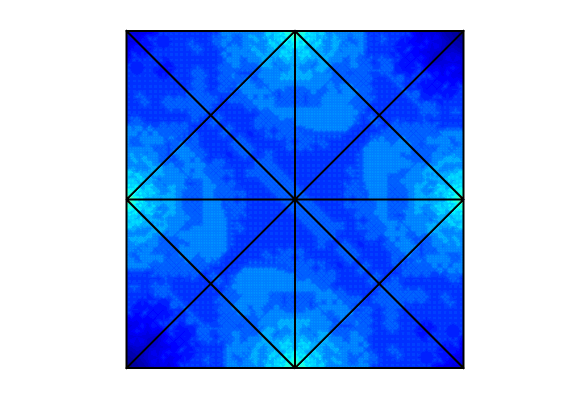}
		\caption{$\#\TT_{12} = 54164$.}
	\end{subfigure}
	\hfill
	\begin{subfigure}[c]{0.22\textwidth}
		\centering
		\includegraphics[trim=70 10 70 10, clip=true, width=\linewidth]{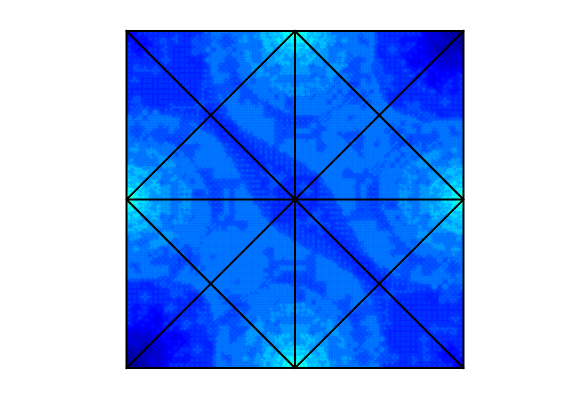}
		\caption{$\#\TT_{13} = 71080$.}
	\end{subfigure}
	\hfill
	\begin{subfigure}[c]{0.06\textwidth}
		\centering
		\includegraphics[trim=350 10 0 10, clip=true, width=\linewidth]{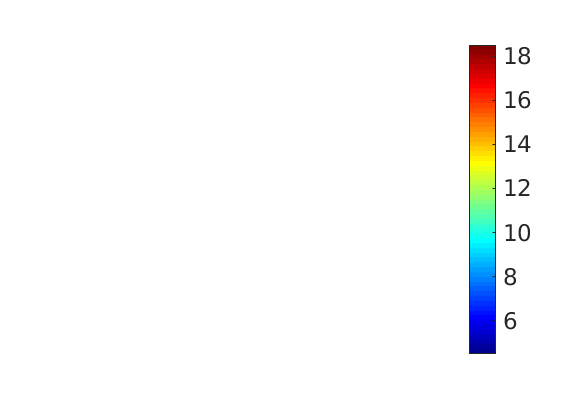}
	\end{subfigure}
	\hfill\hphantom{*}
	\caption{Comparison of meshes resulting from the recursive marking strategy from~\cite{dks16}, \cite{ms2008}, \cite{goafem}, and~\cite{bet2011} (left to right) with polynomial degree $p=1$~(top) and $p=2$~(bottom), and marking parameter $\vartheta = 0.5$.
	The colors show the value of $\log_2(1/|T|)$ for every element $T$, which corresponds to the element's level.}
	\label{fig:mesh-size}
\end{figure}